\documentclass[11pt]{amsart}

\usepackage[T1]{fontenc}
\usepackage{amsmath,amsthm,amsfonts,amssymb,mathtools,mathrsfs}
\usepackage{microtype}
\usepackage{xcolor}
\usepackage{tikz}
\usetikzlibrary{arrows.meta,calc,positioning}

\newcommand{\be}{\begin{equation}}
\newcommand{\ee}{\end{equation}}

\newcommand{\R}{\mathbb{R}}

\newcommand{\SL}{\operatorname{SL}}
\newcommand{\ASL}{\operatorname{ASL}}

\newcommand{\GL}{\operatorname{GL}}
\newcommand{\SO}{\operatorname{SO}}

\newcommand{\vol}{\operatorname{vol}}

\newtheorem{theorem}{Theorem}

\newtheorem{lemma}{Lemma}

\newtheorem{proposition}{Proposition}
\newtheorem{corollary}{Corollary}

\theoremstyle{definition}

\theoremstyle{remark}
\newtheorem{remark}{Remark}
 
\numberwithin{equation}{section}

\def\presuper#1#2%
  {\mathop{}%
   \mathopen{\vphantom{#2}}^{#1}%
   \kern-\scriptspace%
   #2}

\usepackage{enumitem}
\newlist{steps}{enumerate}{1}
\setlist[steps, 1]{label = \textbf{Step \arabic*:}}

\definecolor{dblue}{rgb}{0.09,0.32,0.44} 
\usepackage[colorlinks=true, linkcolor=dblue, citecolor=dblue, urlcolor=blue]{hyperref}

\usepackage[hang,flushmargin]{footmisc}

\newcommand{\Comm}{\operatorname{Comm}}
\newcommand{\Prob}{\mathbb P}
\newcommand{\prperp}{\pi_{\perp}}

\newcommand{\weakto}{\mathrel{\Rightarrow}}
\newcommand{\face}{\ell}
\newcommand{\PGL}{\operatorname{PGL}}

\begin{document}

\author{Christopher Lutsko}
\address{Mathematics Department, University of Houston, 3551 Cullen Blvd,
77204 Houston, USA}
\email{clutsko@uh.edu}

\author{Jens Marklof}
\address{School of Mathematics, University of Bristol, Bristol BS8 1UG, UK}
\email{j.marklof@bristol.ac.uk}

\title[Loss of memory in the periodic Ehrenfest model]{Loss of memory in the periodic Ehrenfest model\\ with small polyhedral scatterers}
\hypersetup{
  pdftitle={Loss of memory in the periodic Ehrenfest model with small polyhedral scatterers},
  pdfauthor={Christopher Lutsko and Jens Marklof},
  pdfsubject={Boltzmann--Grad limit and kinetic transport for the periodic Ehrenfest model},
  pdfkeywords={periodic Ehrenfest model, wind-tree model, Boltzmann--Grad limit,
    polyhedral billiards, homogeneous dynamics, Markov renewal process,
    generalised linear Boltzmann equation}
}

\begin{abstract}
We determine the Boltzmann--Grad limit of the Ehrenfest ``wind-tree'' model with a periodic configuration of polyhedral scatterers. The central result is a loss-of-memory effect in dimension $d\geq 3$ for a generic choice of polyhedron. This yields a Markovian limiting transport process on an extended state space and a generalised linear Boltzmann equation for the time evolution of the particle density.
\end{abstract}



\maketitle

\section{Introduction}

The original wind-tree model, introduced by P.~and T.~Ehrenfest in their classic 1911 encyclopedia article on the conceptual foundations of statistical mechanics \cite{EE1911}, describes a cloud of non-interacting point particles in an array of identically oriented, congruent square scatterers placed at random locations in the two-dimensional plane $\R^2$. It was designed as a more accessible variant of the Lorentz gas \cite{Lorentz1905} (which has spherical scatterers), given the simpler reflection law for the velocities. Gallavotti \cite{Gallavotti1969} proved that in the Boltzmann--Grad limit, the evolution of the particle density is described---in both models---by the linear Boltzmann equation whose collision kernel is the appropriate differential scattering cross section of the single scatterer. Subsequently, Gallavotti's kinetic result was extended to an intermediate kinetic-diffusive time-scale in \cite{LT2020, LT2021} to prove invariance principles for both models.

The objective of the present paper is to study the analogous question for the Ehrenfest model in the periodic setting.  We work in general dimension $d\geq2$, with scatterers located at the points of a covolume-one Euclidean
lattice $\mathscr L$, and consider general polyhedral scatterers $\mathcal P$.
We assume throughout that $\mathcal P\subset\mathbb R^d$ is a bounded open convex
polyhedron with $0\in\mathcal P$ and $\vol(\mathcal P)=1$.    
We will prove that, for $d\geq3$ and a ``generic'' polyhedron, the
Boltzmann--Grad limit is a Markovian random flight on an extended state space.
Besides the usual phase coordinates $(Q,V)\in T^1(\mathbb R^d)$, the extended space
retains the time $\xi$ until the next collision and the labelled impact
parameter $W=(\ell,b)$. We refer to the billiard dynamics lifted to this extended state space as the {\em lifted Ehrenfest flow}; note that a trajectory of the lifted flow is fully determined by initial data in $(Q,V)$. The precise hypothesis on the polyhedron is stated in Section~\ref{sec:nonresonance} in terms of an {\em $\mathscr L$-nonresonance condition}. This nonresonance mechanism has no two-dimensional analogue (Remark~\ref{rem:d2}). 

The main result of this paper is the following.

\begin{theorem}
\label{thm:intro-boltzmann}
Let $d\geq3$, let $\mathscr L$ be a Euclidean lattice of covolume one, and let
$\mathcal P\subset\mathbb R^d$ be as above and $\mathscr L$-nonresonant.  
For random initial data $(Q,V)$ distributed according to an absolutely continuous probability measure, the lifted Ehrenfest flow converges under the Boltzmann--Grad scaling in finite-dimensional
distributions to a time-homogeneous Markov process.  The evolution of the lifted particle density $f$ is given by the generalised linear Boltzmann equation
\begin{equation}\label{GLB}
   \bigl(\partial_t+V\cdot\nabla_Q-\partial_\xi\bigr)f(t,Q,V,\xi,W)
      =[\mathcal C_{\mathcal P}f](t,Q,V,\xi,W)
\end{equation}
where $\mathcal C_{\mathcal P}$ is the collision operator \eqref{eq:CP}.
\end{theorem}

Perhaps surprisingly, the generalised Boltzmann equation for the periodic Ehrenfest model with generic polyhedra is of the same type as the one found for the periodic Lorentz gas \cite{CagliotiGolse2010,MarklofStrom2010,MarklofStrom2011,MarklofStrombergsson2024}. As we will show, its collision term combines the polyhedron's scattering law with the free-path kernel.

The paper is organised as follows.  Section~\ref{sec:geometry} introduces the
periodic Ehrenfest billiard, including the impact and departure
coordinates used throughout the paper.  Section~\ref{sec:phi} recasts a
single free flight as a first-point problem for an affine lattice.  The geometric and
measure-theoretic properties of this construction yield the joint Boltzmann--Grad limit of the
first free path and its labelled impact parameter, valid in every dimension
$d\geq2$.

The passage from one flight to a collision history requires simultaneous
equidistribution of the lattices obtained by unfolding the successive
reflections.  Section~\ref{sec:nonresonance} therefore formulates a
lattice-relative nonresonance condition on the face reflections, proves that
it is generic in both the measure-theoretic and Baire-category senses when
$d\geq3$, and explains the obstruction in dimension $d=2$.  In
Section~\ref{sec:collision-limit}, product equidistribution is combined with
a recursive first-point construction to obtain the impact-resolved law of
every fixed finite collision record.  The factorisation of this law is the
loss-of-memory mechanism behind the limiting dynamics.  

Section~\ref{sec:velocity} pushes the impact-resolved law forward to
velocity and segment variables.  In particular, it records the
differential scattering cross section of a polyhedron and explains why
discarding the impact coordinate does not in general give a closed Markov
description.  Section~\ref{sec:macro} then passes from a fixed number of
collisions to macroscopic continuous time, proving finite-dimensional
convergence and excluding explosion of the limiting flight process.
Finally, Section~\ref{sec:kinetic} constructs the extended
Markov process and proves the full convergence statement. The
Kolmogorov forward equation for the process yields the generalised Boltzmann equation \eqref{GLB}. This completes the proof of Theorem \ref{thm:intro-boltzmann}.

\subsection*{Acknowledgements}
JM's research was supported by EPSRC grant EP/W007010/1 and through access to OpenAI models provided by the ChatGPT for Academic Researchers program. 

\subsection*{Data access statement}
No new data were generated or analysed during this study.

\subsection*{Declaration of AI use}

The authors supplied the main research results and proof outlines for this paper.  OpenAI's Codex was used to assist with mathematical exposition, drafting and clarification of proof arguments, creation of figures, and LaTeX editing. In particular, it corrected and simplified the results in Section \ref{sec:nonresonance} on arithmetic nonresonance, and identified relevant references \cite{Dekker1959,Vorobets2005}. The authors reviewed the AI-assisted revisions and take full responsibility for the work presented in this manuscript.

\section{The dynamical setting}\label{sec:geometry}

Let $d\geq2$, let
$S_1^{d-1}=\{v\in\mathbb R^d:\lVert v\rVert=1\}$, and fix a covolume-one
Euclidean lattice $\mathscr L=\mathbb Z^dM_0$, with
$M_0\in\SL(d,\mathbb R)$.  All vectors are represented as row vectors.
Define the {\em regular boundary} of the polyhedron $\mathcal P$ as a labelled disjoint union
\begin{equation*} \partial_*\mathcal P=\bigsqcup_{\ell=1}^sF_\ell,
\end{equation*}
where $F_1,\ldots,F_s$ are the relative interiors of the codimension-one
faces (or facets) of $\overline{\mathcal P}$ and $n_\ell$ is the outward unit
normal to $F_\ell$.  Write
$dA_\ell$ for $(d-1)$-dimensional Euclidean measure on $F_\ell$.

For $\rho>0$ define the open obstacle set and the completed billiard domain by
\begin{equation*} \mathcal O_\rho=
 \bigcup_{m\in\mathscr L}(m+\rho\mathcal P),
 \qquad
 \mathcal K_\rho=\mathbb R^d\setminus\mathcal O_\rho.
\end{equation*}
We take $\rho$ small enough that these obstacles are disjoint, and write
$T^1(\mathcal K_\rho)$ for the usual completed billiard phase space.  At a
regular collision with $F_\ell$, specular reflection sends the incoming
velocity $v$ to
\begin{equation*} vR_\ell=v-2(v\cdot n_\ell)n_\ell,
\end{equation*}
where $R_\ell=I-2\,{}^tn_\ell n_\ell$. Billiard trajectories and the corresponding initial data $(q_{\rm in},v_0)$ are regular if they encounter only regular collisions. To define the billiard dynamics for all initial data, we agree that trajectories terminate when they hit a lower-dimensional face or graze a face. The initial data leading to such trajectories form a Liouville-null set and, as we shall see, do not affect the Boltzmann--Grad limit.

For regular initial data $(q_{\rm in},v_0)$, let $t_j$ be the
collision times, $\tau_j=t_j-t_{j-1}$ the free-flight times, $q_j$ the
collision points, and $v_j$ the post-collision velocities. We also put $t_0=0$.
For an initial free phase point $(q,v)$ we write
\begin{equation*} \tau_1(q,v;\rho)=
 \inf\{t>0:q+tv\in\partial\mathcal O_\rho\},
\end{equation*}
with $\inf\varnothing=\infty$.
If the $j$th collision is at the point $q_j\in m_j+\rho F_{\ell_j}$, set
\begin{equation*} x_j^{(\rho)}=\rho^{-1}(q_j-m_j)\in F_{\ell_j},
 \qquad v_j=v_{j-1}R_{\ell_j}.
\end{equation*}
If $\tau_j=\infty$, or if the dynamics for nonregular initial data terminates at time $t_j$, we formally declare all subsequent flight times, collision points, and velocities undefined.

Let $e_1=(1,0,\ldots,0)$.  Fix a measurable map
$K:S_1^{d-1}\to\SO(d)$, smooth away from one direction
$v_{\rm s}$,
such that
\begin{equation*} vK(v)=e_1.
\end{equation*}
For an explicit construction, see \cite{MarklofStrom2010}.  Let $\omega$
denote the uniform probability measure on $S_1^{d-1}$, put $E=e_1^\perp$, and write $\prperp$ and
$\vol_E$ for orthogonal projection and Lebesgue measure on $E$.
For $v\in S_1^{d-1}$ define the projected shadow and the impact shadow by
\begin{equation*} C_v=\prperp\bigl(\mathcal P K(v)\bigr),
 \qquad B_v=-C_v\subset E.
\end{equation*}
The sign in $B_v=-C_v$ is a coordinate convention.  Put
\begin{equation*} A_{\mathcal P}=\sup_{v\in S_1^{d-1}}\vol_E(B_v)<\infty.
\end{equation*}
The family $\{B_v\}$ has compact closure in the Hausdorff topology, with
inradii uniformly bounded below and outradii uniformly bounded above.  For
$b\in B_v$ define
\begin{equation*}
 t_v^-(b)=\inf\{t\in\mathbb R:te_1-b\in\mathcal P K(v)\},
 \qquad
 x_v^-(b)=\bigl(t_v^-(b)e_1-b\bigr)K(v)^{-1}.
\end{equation*}
If $x_v^-(b)$ lies in the regular incoming boundary
\begin{equation*}
 \partial_-\mathcal P(v)=\bigsqcup_{v\cdot n_\ell<0}F_\ell
\end{equation*}
it determines a unique face label $\face_v^-(b)\in\{1,\ldots,s\}$; on the remaining null set of $b$, extend $\face_v^-(b)$ measurably by choosing the least index among the adjacent incoming faces. In this way, the face cells
\begin{equation*} B_{v,\ell}=\{b\in B_v:\face_v^-(b)=\ell\}
\end{equation*}
partition $B_v$.  Define the impact space at velocity $v$ by

\begin{equation*} \mathcal B_v=\bigsqcup_{\ell=1}^s(\{\ell\}\times B_{v,\ell}).
\end{equation*}
Then $W=(\ell,b)\in\mathcal B_v$ is the \emph{labelled impact parameter}. For regular $b$, it represents the point $x_v^-(b)\in F_\ell$; on the null cell boundaries, $\ell$ is the measurable tie-breaking label as specified above.
The global impact bundle
\begin{equation*} \mathcal B=\{(v,W):v\in S_1^{d-1},\ W\in\mathcal B_v\}
\end{equation*}
is given the Borel structure and cellwise topology inherited from
$S_1^{d-1}\times\{1,\ldots,s\}\times E$.

If a particle with velocity $v$ has next impact parameter $b$, define
\begin{align*} \mathcal R(v,b)&=vR_{\face_v^-(b)},\\
 \mathcal Z(v,b)&=\prperp\bigl(x_v^-(b)K(\mathcal R(v,b))\bigr).
\end{align*}
If $b\in B_{v,\ell}$ is a regular impact and $u=vR_\ell$, then
$\mathcal Z(v,b)\in\prperp(F_\ell K(u))$.  Modulo the null cell boundaries, the facewise map
\begin{equation*} B_{v,\ell}\longrightarrow \prperp(F_\ell K(u)),\qquad b\longmapsto\mathcal Z(v,b),
\end{equation*}
is measure preserving: the source and target are projections of the same face and,
since $u\cdot n_\ell=-v\cdot n_\ell$, their flux measures are
$(-v\cdot n_\ell)dA_\ell$.
For a regular collision record we therefore put
\begin{equation} \xi_j^{(\rho)}=\rho^{d-1}\tau_j,
 \quad b_j^{(\rho)}=-\prperp\bigl(x_j^{(\rho)}K(v_{j-1})\bigr),
 \quad W_j^{(\rho)}=(\ell_j,b_j^{(\rho)}),
 \quad z_j^{(\rho)}=\mathcal Z(v_{j-1},b_j^{(\rho)}).
 \label{reccollrec}
\end{equation}
Figure~\ref{fig:impact-geometry} summarises these coordinates.

\begin{figure}[t]
\centering
\begin{tikzpicture}[
    scale=1.06,
    >=Latex,
    poly/.style={draw=black,fill=blue!7,line width=.9pt,line join=round},
    active face/.style={draw=blue!65!black,line width=1.45pt,line cap=round},
    ray/.style={-Latex,line width=1.15pt},
    normal/.style={-Latex,line width=.9pt},
    axis/.style={-Latex,draw=black!65,line width=.8pt},
    parameter/.style={-Latex,draw=blue!65!black,line width=1.0pt},
    projection/.style={densely dotted,draw=black!70,line width=.85pt,
      line cap=round}]
   \def\polysection{(-1.35,-.65)--(-1.05,1.15)--(1.05,1.15)--
    (1.40,-.25)--(.55,-1.25)--(-.65,-1.15)--cycle}
  \def\activeface{(-1.05,1.15)--(1.05,1.15)}

  \begin{scope}[yshift=4.55cm]
    \coordinate (xp) at (0,1.15);
    \path[poly] \polysection;
    \draw[active face] \activeface;
    \draw[ray] (-.866,1.65)--(xp)
      node[pos=.32,above left=-1pt] {$v$};
    \draw[ray] (xp)--(.866,1.65)
      node[pos=.72,above right=-1pt] {$u$};
    \draw[normal] (xp)--(0,2.08) node[pos=.78,right=2pt] {$n_\ell$};
    \fill (xp) circle[radius=2.2pt];
    \node[anchor=north west,inner sep=1pt] at (.14,1.04) {$x$};
    \node[blue!65!black] at (-.58,.98) {$F_\ell$};
    \node at (0,-.35) {$\mathcal P$};
    \node[font=\small] at (0,-1.58) {physical section};
  \end{scope}

  \begin{scope}[xshift=-3.0cm]
    \coordinate (Ov) at (0,0);
    \draw[axis] (-1.85,0)--(1.80,0) node[right] {$e_1$};
    \draw[axis] (0,-1.48)--(0,2.12) node[above] {$E$};
    \begin{scope}[rotate=30]
      \path[poly] \polysection;
      \draw[active face] \activeface;
      \coordinate (xv) at (0,1.15);
      \draw[ray] (-.866,1.65)--(xv);
      \draw[ray] (xv)--(.866,1.65);
      \node at (0,-.35) {$\mathcal P K(v)$};
    \end{scope}
    \node[font=\small,anchor=east] at (-1.78,1.00) {$vK(v)=e_1$};
    \node[font=\small,anchor=west] at (.26,1.78) {$uK(v)$};
    \coordinate (bv) at (Ov |- xv);
    \draw[projection] (xv)--(bv);
    \draw[parameter] (Ov)--(bv) node[pos=.52,right=2pt] {$-b$};
    \fill (xv) circle[radius=2.2pt];
    \node[font=\small,anchor=south east,inner sep=1pt]
      (xvlabel) at (-.92,1.48) {$xK(v)$};
    \draw[draw=black!60,line width=.45pt]
      (xv)--(xvlabel.south east);
    \node[font=\small] at (0,-1.72) {incoming chart};
  \end{scope}

  \begin{scope}[xshift=3.0cm]
    \coordinate (Ou) at (0,0);
    \draw[axis] (-1.80,0)--(1.85,0) node[right] {$e_1$};
    \draw[axis] (0,-1.48)--(0,2.12) node[above] {$E$};
    \begin{scope}[rotate=-30]
      \path[poly] \polysection;
      \draw[active face] \activeface;
      \coordinate (xu) at (0,1.15);
      \draw[ray] (-.866,1.65)--(xu);
      \draw[ray] (xu)--(.866,1.65);
      \node at (0,-.35) {$\mathcal P K(u)$};
    \end{scope}
    \node[font=\small,anchor=east] at (-.26,1.98) {$vK(u)$};
    \node[font=\small,anchor=west] at (1.76,1.00) {$uK(u)=e_1$};
    \coordinate (zu) at (Ou |- xu);
    \draw[projection] (xu)--(zu);
    \draw[parameter] (Ou)--(zu) node[pos=.52,left=2pt] {$z$};
    \fill (xu) circle[radius=2.2pt];
    \node[font=\small,anchor=south west,inner sep=1pt]
      (xulabel) at (.92,1.48) {$xK(u)$};
    \draw[draw=black!60,line width=.45pt]
      (xu)--(xulabel.south west);
    \node[font=\small] at (0,-1.72) {outgoing chart};
  \end{scope}
\end{tikzpicture}
\caption{A planar section of one regular collision in physical coordinates
and in the two straightened charts.  Every panel shows the same polygonal
section, collision point, and velocity pair, rotated together.  In the
incoming chart $vK(v)=e_1$, and the dotted segment projects $xK(v)$
orthogonally onto $E$, giving $-b$.  In the outgoing chart $uK(u)=e_1$, and
the corresponding projection gives $z$.}
\label{fig:impact-geometry}
\end{figure}
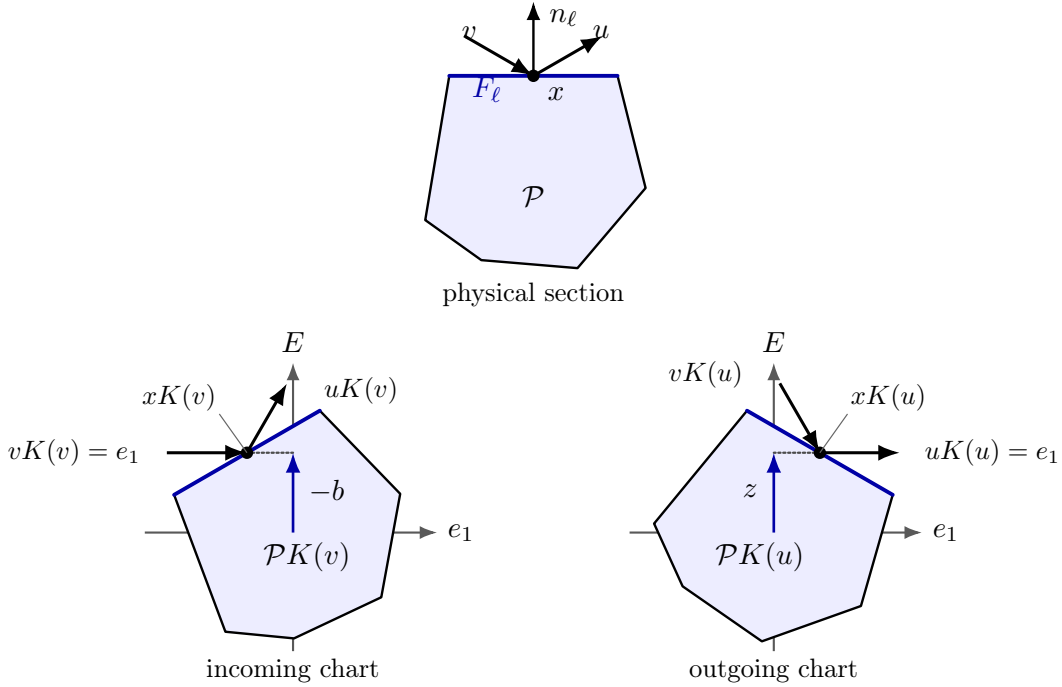

Similarly, for $z\in C_v$ put
\begin{equation*}
 t_v^+(z)=\sup\{t\in\mathbb R:te_1+z\in\mathcal P K(v)\},
 \qquad x_v^+(z)=(t_v^+(z)e_1+z)K(v)^{-1}.
\end{equation*}
For almost every $z$, this point lies in the regular outgoing boundary
\begin{equation*} \partial_+\mathcal P(v)=\bigsqcup_{v\cdot n_\ell>0}F_\ell.
\end{equation*}
Furthermore, define
\begin{equation*} C_{v,\ell}=\{z\in C_v:x_v^+(z)\in F_\ell\}.
\end{equation*}
Thus, modulo the null cell boundaries, the target of the preceding facewise map is $C_{u,\ell}$, and
$b$ and $z$ are the incoming impact and outgoing departure parameters.
Our sign convention gives the terminal lattice point $y=\xi e_1+b+z$, cf.~Figure~\ref{fig:collision-tube}.

The projected measures are the billiard flux measures.  Namely,
\begin{equation}\label{eq:flux-in}
 db=(-v\cdot n_\ell)\,dA_\ell(x),
 \qquad x\in F_\ell,\quad v\cdot n_\ell<0,
\end{equation}
and on a regular outgoing face,
\begin{equation}\label{eq:flux-out}
 dz=(v\cdot n_\ell)\,dA_\ell(x),
 \qquad x\in F_\ell,\quad v\cdot n_\ell>0.
\end{equation}

\section{Limit law for the first collision}\label{sec:phi}

We now adapt the first-collision construction of
\cite[Sections~4 and~7]{MarklofStrom2010}.  Put
\begin{equation*}
 \ASL(d,\mathbb R)=\SL(d,\mathbb R)\ltimes\mathbb R^d,
 \qquad
 \ASL(d,\mathbb Z)=\SL(d,\mathbb Z)\ltimes\mathbb Z^d,
\end{equation*}
with multiplication law $(M_1,\eta_1)(M_2,\eta_2)=(M_1M_2,\eta_1 M_2+\eta_2)$,
and define
\begin{equation*} X=\ASL(d,\mathbb Z)\backslash\ASL(d,\mathbb R),
 \qquad X_q=\Gamma(q)\backslash\SL(d,\mathbb R),
\end{equation*}
where
$\Gamma(q)=\{\gamma\in\SL(d,\mathbb Z):\gamma\equiv I\pmod q\}$.
If $\alpha\in\mathbb Q^d$, let $q_\alpha$ be its least positive denominator
and put $X_\alpha=X_{q_\alpha}$; for $\alpha\in\mathbb R^d\setminus\mathbb Q^d$, put $X_\alpha=X$. Write
$\mu_\alpha$ for Haar probability measure on $X_\alpha$.
The affine lattice, as a point set, represented by
$g\in X_\alpha$ is
\begin{equation*} \Lambda_\alpha(g)=
 \begin{cases}
  (\mathbb Z^d+\alpha)M,
     &g=\Gamma(q_\alpha)M\in X_{q_\alpha},\quad
       \alpha\in q_\alpha^{-1}\mathbb Z^d,\\
  \mathbb Z^dM+\eta,
     &g=\ASL(d,\mathbb Z)(M,\eta)\in X,\quad \alpha\notin\mathbb Q^d.
 \end{cases}
\end{equation*}
Write $X_\alpha(y)=\{g:y\in\Lambda_\alpha(g)\}$ and
$\nu_y^{(\alpha)}$ for the disintegrated probability measure on this fibre;
see \cite[Section~7]{MarklofStrom2010}.

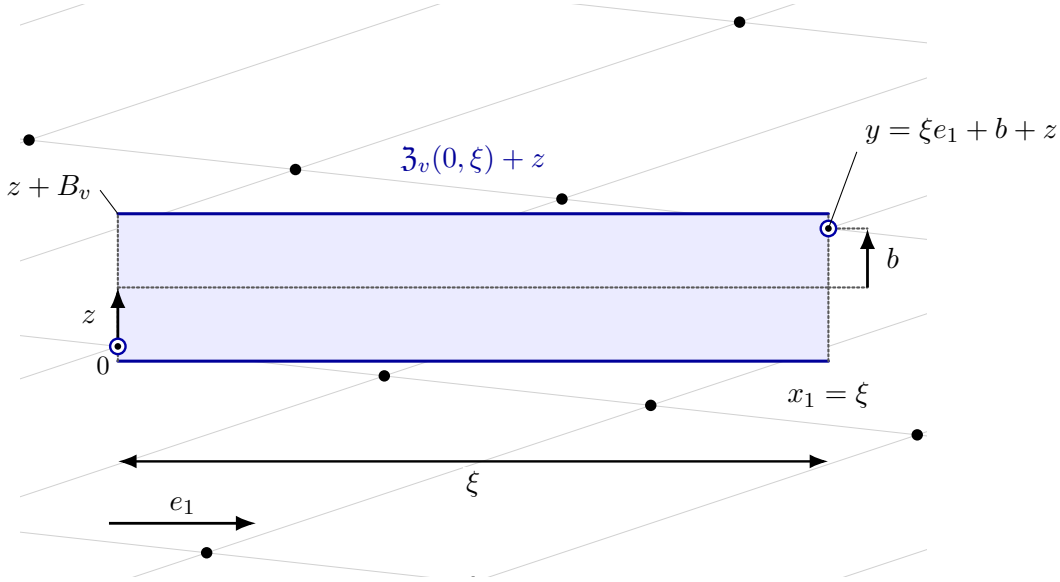
\begin{figure}[t]
\centering
\begin{tikzpicture}[
    x=2.35cm,y=3.9cm,>=Latex,
    tube edge/.style={draw=blue!60!black,line width=1.15pt,line cap=round},
    guide/.style={densely dotted,draw=black!65,line width=.8pt,
      line cap=round},
    vector/.style={-Latex,line width=1.05pt},
    dimension/.style={<->,line width=.9pt},
    lattice grid/.style={draw=black!18,line width=.35pt}]
  
    \begin{scope}
    \clip (-.55,-.78) rectangle (4.55,1.16);
    \foreach \m in {-5,...,5}{
      \foreach \n in {-5,...,5}{
        \pgfmathsetmacro{\px}{2.5*\m+1.5*\n}
        \pgfmathsetmacro{\py}{.5*\m-.1*\n}
        \draw[lattice grid] (\px,\py)--({\px+2.5},{\py+.5});
        \draw[lattice grid] (\px,\py)--({\px+1.5},{\py-.1});
      }
    }
  \end{scope}

  \fill[blue!8] (0,-.05) rectangle (4,.45);
  \draw[tube edge] (0,-.05)--(4,-.05);
  \draw[tube edge] (0,.45)--(4,.45);
  \draw[guide] (0,.20)--(4,.20);
  \draw[guide] (0,-.05)--(0,.45);
  \draw[guide] (4,-.05)--(4,.45);

    \begin{scope}
    \clip (-.55,-.78) rectangle (4.55,1.16);
    \foreach \m in {-5,...,5}{
      \foreach \n in {-5,...,5}{
        \pgfmathsetmacro{\px}{2.5*\m+1.5*\n}
        \pgfmathsetmacro{\py}{.5*\m-.1*\n}
        \fill[black] (\px,\py) circle[radius=2.15pt];
      }
    }
  \end{scope}

  \node[above=3pt,blue!60!black,fill=white,inner sep=1.5pt] at (2,.54)
    {$\mathfrak Z_v(0,\xi)+z$};
  \draw[thin] (0,.45)--(-.10,.53);
  \node[anchor=east,fill=white,inner sep=1.5pt] at (-.12,.54)
    {$z+B_v$};

    \draw[vector] (0,0)--(0,.20)
    node[midway,left=4pt] {$z$};

  \draw[guide] (4,.20)--(4.22,.20);
  \draw[guide] (4,.40)--(4.22,.40);
  \draw[vector] (4.22,.20)--(4.22,.40)
    node[midway,right=3pt] {$b$};

    \draw[blue!65!black,line width=.9pt,fill=white]
    (0,0) circle[radius=2.9pt];
  \fill (0,0) circle[radius=1.25pt];
  \node[font=\small,anchor=north east,inner sep=1.5pt] at (-.02,-.02) {$0$};
  \draw[blue!65!black,line width=.9pt,fill=white]
    (4,.40) circle[radius=2.9pt];
  \fill (4,.40) circle[radius=1.25pt];
  \draw[thin] (4,.40)--(4.16,.66)
    node[above right=-1pt] {$y=\xi e_1+b+z$};
  \node[below=4pt] at (4,-.05) {$x_1=\xi$};

  \draw[vector] (-.05,-.60)--(.78,-.60)
    node[midway,above=2pt,fill=white,inner sep=1pt] {$e_1$};
  \draw[dimension] (0,-.39)--(4,-.39)
    node[midway,below=2pt,fill=white,inner sep=1pt] {$\xi$};
\end{tikzpicture}
\caption{A planar schematic of the collision tube.  The black dots represent the rotated and stretched lattice.  The source centre $0$ lies on the left end
section and is joined by the vector $z$ to the dotted longitudinal reference
line.  The terminal centre $y$ lies on the right end section $x_1=\xi$, with
terminal impact parameter $b$.}
\label{fig:collision-tube}
\end{figure}

Let $B\subset E$ be a bounded open convex set with nonempty interior and
Lebesgue-null boundary.  For $c_1<c_2$ write
\begin{equation*} \mathfrak Z_B(c_1,c_2)=
 \{x_1e_1+u:c_1<x_1<c_2,\ u\in B\}.
\end{equation*}
For $\xi>0$, $b\in B$, and $z\in E$, put $y=\xi e_1+b+z$ and define
\begin{equation}\label{eq:Phi-definition} \Phi_{\alpha,B}(\xi,b,z)=
 \nu_y^{(\alpha)}\!\left(
 \left\{g\in X_\alpha(y):
 \Lambda_\alpha(g)\cap(\mathfrak Z_B(0,\xi)+z)=\varnothing
 \right\}\right).
\end{equation}
If we set $B=B_v$, we obtain the collision tube $\mathfrak Z_v(c_1,c_2):=\mathfrak Z_{B_v}(c_1,c_2)$ and kernel $\Phi_{\alpha,v}(\xi,b,z):=\Phi_{\alpha,B_v}(\xi,b,z)$. Here $z$ is the departure parameter, $b$ the terminal impact parameter, and $y$ the transformed centre of the terminal scatterer. Figure~\ref{fig:collision-tube} depicts the collision-to-collision case $\alpha=0$.  The source centre $0$ and terminal centre $y$ lie on the two end
sections of the tube.

\begin{lemma}\label{lem:general-shadow}
{\rm (i)} Given any $z\in E$, for $\mu_\alpha$-almost every $g$, the set
$\Lambda_\alpha(g)\cap(\mathfrak Z_B(0,\infty)+z)$ has a unique point with
least positive first coordinate; write it as
$y_B(g)=\xi_B(g)e_1+b_B(g)+z$.  

{\rm (ii)} For $z\in E$ and every Borel set
$A\subset\mathbb R_{>0}\times B$,
\begin{equation}\label{eq:first-point-law}
 \mu_\alpha\{g:(\xi_B(g),b_B(g))\in A\}
 =\int_A\Phi_{\alpha,B}(\xi,b,z)\,d\xi\,db.
\end{equation}

{\rm (iii)} For $z\in E$,
\begin{equation}\label{eq:Phi-normalisation}
 \int_0^\infty\int_B
 \Phi_{\alpha,B}(\xi,b,z)\,db\,d\xi=1.
\end{equation}

{\rm (iv)} The map $(\xi,b,z)\mapsto\Phi_{\alpha,B}(\xi,b,z)$ is Borel measurable.
\end{lemma}

\begin{proof}
For a bounded Borel set $A\subset\mathbb R_{>0}\times B$, bounded away from
$\xi=0$, put
\begin{align*} N^{\rm first}_{A}(g)
 ={}&\sum_{\substack{y=\xi e_1+b+z\in\Lambda_\alpha(g)\\(\xi,b)\in A}}
 \mathbf 1\{\Lambda_\alpha(g)\cap
 (\mathfrak Z_B(0,\xi)+z)=\varnothing\}.
\end{align*}
Outside the set of terminal ties and tube boundary points this is the indicator of the
event that the first point has coordinates in $A$.
Firstly, there are no tube boundary points almost surely, by Siegel's volume formula.  To
exclude ties, thicken the terminal hyperplane and apply
\cite[Lemmas~7.12--7.13]{MarklofStrom2010}; the resulting two-point bound tends
to zero with the thickness.
The foliation formulas
\cite[Propositions~7.3 and~7.10]{MarklofStrom2010} give
\begin{align}
 \int N^{\rm first}_{A}(g)\,d\mu_\alpha(g)
 ={}&\int_A
 \nu_{\xi e_1+b+z}^{(\alpha)}
 \{g:\Lambda_\alpha(g)\cap(\mathfrak Z_B(0,\xi)+z)=\varnothing\}
 \,d\xi\,db. \label{eq:positive-bin-disintegration}
\end{align}
This yields \eqref{eq:first-point-law} for bounded test sets.  The translated-cylinder
bound of \cite[Section~8.4]{MarklofStrom2010}, applied to a ball contained in
$B$, shows that the first point is finite almost surely.  Hence we obtain
\eqref{eq:first-point-law} for general Borel sets and 
\eqref{eq:Phi-normalisation}.  Finally, the incidence and avoidance conditions
are Borel, as is the kernel $y\mapsto\nu_y^{(\alpha)}$; hence
$\Phi_{\alpha,B}$ is Borel measurable.
\end{proof}

Define the \emph{Hausdorff distance} between nonempty compact sets $A,A'\subset E$ by
\begin{equation*} d_{\rm H}(A,A')=
 \max\left\{\sup_{x\in A}\operatorname{dist}(x,A'),
            \sup_{x'\in A'}\operatorname{dist}(x',A)\right\}.
\end{equation*}
For bounded open convex sets we apply this distance to their closures. The inradius of $B$ is the
supremal radius of a ball contained in $B$, and its outradius the infimal
radius of an origin-centred ball containing it.

\begin{lemma}
\label{lem:shadow-stability}

{\rm (i)} For nonempty bounded open convex sets $B,B'\subset E$, $z,z'\in E$, and
$T<\infty$,
\begin{align}
 &\mu_\alpha\!\left\{g:\Lambda_\alpha(g)\cap
 \left[(\mathfrak Z_B(0,T)+z)\mathbin\triangle
 (\mathfrak Z_{B'}(0,T)+z')\right]\ne\varnothing\right\}\notag\\
 &\hspace{35mm}\leq
 \vol\!\left((\mathfrak Z_B(0,T)+z)\mathbin\triangle
 (\mathfrak Z_{B'}(0,T)+z')\right).
 \label{eq:first-point-stability}
\end{align}

{\rm (ii)} If $B_k\to B$ in the Hausdorff topology and $z_k\to z$, the
corresponding first-point laws converge against every bounded continuous
test function on $\mathbb R_{>0}\times E$ supported in $0<\xi\leq T$.

{\rm (iii)} If $B$ ranges over a Hausdorff-compact convex family whose inradii are
uniformly bounded below by a positive constant and whose outradii are
uniformly bounded above, and $z$ ranges over a compact subset of $E$, then
the laws are uniformly tight; explicitly,
\begin{equation*} \lim_{T\to\infty}\ \sup_{B,z}\mu_\alpha\{\xi_B>T\}=0.
\end{equation*}
\end{lemma}

\begin{proof}
Part (i) is Siegel's formula applied to the symmetric difference.  Coupling
the two first-point constructions on the same lattice gives (ii).  Finally,
every member of the family contains a ball of uniformly positive radius, so
the bound of \cite[Section~8.4]{MarklofStrom2010} gives (iii).
\end{proof}

\begin{remark}\label{rem:Phi-covariance}
Write $\mathrm O(E)$ and $\mathrm{SO}(E)$ for the orthogonal and
orientation-preserving orthogonal groups of the Euclidean space $E$.  For
$O\in\mathrm{SO}(E)$,
\begin{equation}\label{eq:Phi-covariance}
 \Phi_{\alpha,BO}(\xi,bO,zO)=\Phi_{\alpha,B}(\xi,b,z).
\end{equation}
The same identity holds for $O\in\mathrm O(E)$ in the affine-Haar case and
when $\alpha=0$.  Right multiplication by $\operatorname{diag}(1,O)$ proves
the first assertion.  When $\det O=-1$, compose on the left with any
$J\in\GL(d,\mathbb Z)$ of determinant $-1$; this proves the remaining cases.
\end{remark}

Apply Lemma~\ref{lem:general-shadow} with $B=B_v$.  Formula
\eqref{eq:Phi-definition} can be viewed as the
polyhedral analogue of \cite[Eq.~(4.14)]{MarklofStrom2010}.  For
$\alpha\notin\mathbb Q^d$ the kernel is independent of $\alpha$ and $z$;
we then write $\Phi_v(\xi,b)$.  Extended by zero for $b\notin B_v$, the map
$(v,\xi,b,z)\mapsto\Phi_{\alpha,v}(\xi,b,z)$ is Borel, as are its
restrictions to the face cells.  Lemma~\ref{lem:shadow-stability} gives the
required integrated continuity in $v$.

Fix a microscopic base point $q_\star\in\mathbb R^d$ and let
$q_{\rho,\beta}(v)=q_\star+\rho\beta(v)$, where $\beta:S_1^{d-1}\to\mathbb R^d$ is continuous. When
$q_\star\in\mathscr L$ assume the admissibility condition
\begin{equation}\label{eq:beta-admissibility}
 \bigl(\beta(v)+\mathbb R_{>0}v\bigr)
 \cap\overline{\mathcal P}=\varnothing
 \qquad(v\in S_1^{d-1});
\end{equation}
this excludes an immediate return to the source scatterer.  Let $\alpha$ be
the chosen representative of the class
$-q_\star M_0^{-1}+\mathbb Z^d$, taking $\alpha=0$ when
$q_\star\in\mathscr L$, and put
\begin{equation*} z_0(v)=\prperp\bigl(\beta(v)K(v)\bigr).
\end{equation*}
The family $\{B_v\}$ has Hausdorff-compact closure with the
required uniform inradius and outradius bounds, while $\{z_0(v)\}$ has compact
closure in $E$; hence Lemma~\ref{lem:shadow-stability}(iii) applies.
If the flight starts on a scatterer, $z_0(v)$ is its departure parameter.

Let $\mathcal A_{1,\rho}$ be the set of velocities for which
the first collision exists and is regular. On this set define
\begin{equation*}
 \mathcal I_{1,\rho}(v)=
 \bigl(\rho^{d-1}\tau_1(q_{\rho,\beta}(v),v;\rho),
       (v,W_1^{(\rho)})\bigr)
 \in\mathbb R_{>0}\times\mathcal B .
\end{equation*}
The following theorem identifies the limiting integrals of
this quantity as $\rho\to0$; taking $G\equiv1$ also implies that
$\lambda(\mathcal A_{1,\rho}^{c})\to0$.

\begin{theorem}\label{thm:one-flight}
Let $\lambda$ be a Borel probability measure on
$S_1^{d-1}$ absolutely continuous with respect to $\omega$. Then, for every
$G\in C_b(\mathbb R_{>0}\times\mathcal B)$,
\begin{align}
 &\lim_{\rho\to 0}\int_{\mathcal A_{1,\rho}}
 G\bigl(\mathcal I_{1,\rho}(v)\bigr)\,d\lambda(v) \notag\\
 &\quad=
 \sum_{\ell=1}^s\int_{S_1^{d-1}}\int_0^\infty\int_{B_{v,\ell}}
 G(\xi,(v,(\ell,b)))
 \Phi_{\alpha,v}(\xi,b,z_0(v))\,db\,d\xi\,d\lambda(v).
 \label{eq:one-flight-limit}
\end{align}
\end{theorem}

\begin{proof}
We follow \cite[Theorem~4.4]{MarklofStrom2010}.
Write $q_{\rho,\beta}(v)+\tau v=m+\rho x$, where $m\in\mathscr L$ is the
centre of the first scatterer and $x\in\partial_*\mathcal P$ is the local
boundary point.  Put $\xi=\rho^{d-1}\tau$, $b=-\prperp(xK(v))$, and
$z=\prperp(\beta(v)K(v))$.  After applying
\begin{equation*} D(\rho)=\operatorname{diag}(\rho^{d-1},\rho^{-1},\ldots,\rho^{-1}),
\end{equation*}
the transformed centre set is
\begin{equation*} (\mathscr L-q_\star)K(v)D(\rho)
 =(\mathbb Z^d+\alpha)M_0K(v)D(\rho).
\end{equation*}
The terminal centre satisfies the exact identity
\begin{equation}\label{eq:terminal-centre-finite-rho}
 (m-q_\star)K(v)D(\rho)
 =\bigl[\xi+\rho^d(\beta(v)-x)\cdot v\bigr]e_1+b+z.
\end{equation}
More generally, a transformed lattice point is the centre of an obstacle
intersecting the open flight segment precisely when it belongs to the region
\begin{equation}\label{eq:finite-rho-collision-set}
 \mathfrak C_{\rho,v}(\xi,z):=
 \left\{[s+\rho^d(\beta(v)-p)\cdot v]e_1+z
 -\prperp(pK(v)):0<s<\xi,\ p\in\mathcal P\right\}.
\end{equation}
Since $\beta$ is bounded and $\mathcal P$ is bounded, there is a constant $C$ such that, uniformly in $v$ and whenever $\xi>2C\rho^d$,
\begin{align}
 \mathfrak Z_v(C\rho^d,\xi-C\rho^d)+z
 \subseteq{}&\ \mathfrak C_{\rho,v}(\xi,z)
 \subseteq\mathfrak Z_v(-C\rho^d,\xi+C\rho^d)+z.
 \label{eq:tube-sandwich}
\end{align}
If $q_\star\in\mathscr L$, we delete the zero vector; admissibility
\eqref{eq:beta-admissibility} excludes a return to the source scatterer.

Now take $G$ supported in $\varepsilon\leq\xi\leq T$ and
away from the face-cell
boundaries.  The equidistribution and lattice-counting results of
\cite[Sections~5--6 and Theorem~3.9]{MarklofStrom2010}, together with
\eqref{eq:tube-sandwich}, give convergence of the transformed affine lattice.
The first-point map is continuous off the null set of ties and boundary
points, and Lemma~\ref{lem:general-shadow} identifies its limiting law with
$\Phi_{\alpha,v}(\xi,b,z)\,d\xi\,db$.  The grazing set is $\omega$-null, and the cell boundaries have zero flux measure.  The discarded short flights have mass
$O(A_{\mathcal P}\varepsilon)$ by Siegel's formula, while
Lemma~\ref{lem:shadow-stability} controls the long flights.
The preceding argument first gives convergence on the
regular truncated region. Apply it to continuous cutoffs increasing to one
as $\varepsilon\to0$, $T\to\infty$, and the neighbourhoods of the
grazing and face-cell-boundary strata shrink. By
\eqref{eq:Phi-normalisation}, the limiting cutoff integrals tend to one, and
hence $\lambda(\mathcal A_{1,\rho}^{c})\to0$. The same estimates remove the
cutoffs for bounded continuous $G$, which proves \eqref{eq:one-flight-limit}.
\end{proof}

We will generalise Theorem~\ref{thm:one-flight} to $n$ collisions in Theorem~\ref{thm:impact-micro}.

\section{Arithmetic nonresonance of the faces}\label{sec:nonresonance}

Assume throughout this section that $d\geq3$. For
$A\in\GL(d,\mathbb R)$, put
\begin{equation*} [A]=\{cA:c\in\mathbb R^\times\}\in
 \PGL(d,\mathbb R)=\GL(d,\mathbb R)/(\mathbb R^\times I).
\end{equation*}
Thus $[A]\in\PGL(d,\mathbb Q)$ if some nonzero scalar multiple of $A$ is
rational.  Two subgroups are \emph{commensurable} if their intersection has
finite index in each.  The commensurator of $\Gamma\leq G$ is
\begin{equation}\label{eq:commensurator-definition}
 \Comm_G(\Gamma)=
 \{g\in G:\Gamma\text{ and }g\Gamma g^{-1}
                 \text{ are commensurable}\}.
\end{equation}
For the integer special linear group,
\begin{equation}\label{eq:commensurator}
 \mathscr S_d=
 \Comm_{\SL(d,\mathbb R)}\bigl(\SL(d,\mathbb Z)\bigr)
 =\left\{cT\in\SL(d,\mathbb R):
 T\in\GL(d,\mathbb Q),\ c\in\mathbb R^\times\right\}.
\end{equation}
Thus, for $A\in\SL(d,\mathbb R)$,
$A\in\mathscr S_d$ if and only if $[A]\in\PGL(d,\mathbb Q)$.

Let
\begin{equation*} \mathcal W_s=\langle a_1,\ldots,a_s\mid a_\ell^2=e\rangle
 =\mathbb Z_2*\cdots*\mathbb Z_2
\end{equation*}
be the free product on the formal face involutions.  A word is {\em reduced} if
adjacent letters are distinct.  If $w=\ell_1\cdots \ell_r$, put
$R_w=R_{\ell_1}\cdots R_{\ell_r}$ and $|w|=r$.

The labelled polyhedron $\mathcal P$ is called
\emph{$\mathscr L$-nonresonant} if
\begin{equation}\label{eq:NR}
 [M_0R_wM_0^{-1}]\notin\PGL(d,\mathbb Q)
 \quad\text{for every nonempty reduced }w\in\mathcal W_s.
\end{equation}

Replacing $M_0$ by $\gamma M_0$, $\gamma\in\SL(d,\mathbb Z)$, conjugates each class in \eqref{eq:NR} by a
rational projective class and does not affect the condition. The condition
excludes parallel faces, since their two reflections have product $I$, and
makes the $[R_\ell]$ generate the free product.

Let $J=\operatorname{diag}(-1,1,\ldots,1)$.
(In odd dimension one may equivalently use $J=-I$.)

\begin{lemma}\label{lem:pairwise}
For a reduced word $u$ put
\begin{equation*} M_u^\circ=J^{|u|}M_0R_u\in\SL(d,\mathbb R).
\end{equation*}
Then $\mathbb Z^dM_u^\circ=\mathbb Z^dM_0R_u$, so $M_u^\circ$ is an
oriented basis of the unfolded lattice associated with $u$.  If
$\mathcal P$ is $\mathscr L$-nonresonant and $u\ne v$, then the arithmetic
lattices
\begin{equation*} \Gamma_u=(M_u^\circ)^{-1}\SL(d,\mathbb Z)M_u^\circ,
 \qquad
 \Gamma_v=(M_v^\circ)^{-1}\SL(d,\mathbb Z)M_v^\circ
\end{equation*}
are incommensurable.
\end{lemma}

\begin{proof}
Let $w$ be the reduced representative of $uv^{-1}$ in $\mathcal W_s$.
Since $u\ne v$, the word $w$ is
nonempty.  Since $J$ has rational entries,
\begin{align*} &[(M_u^\circ)(M_v^\circ)^{-1}]\in\PGL(d,\mathbb Q)\\
 &\qquad\Longleftrightarrow
 [M_0R_uR_v^{-1}M_0^{-1}]\in\PGL(d,\mathbb Q)\\
 &\qquad\Longleftrightarrow
 [M_0R_wM_0^{-1}]\in\PGL(d,\mathbb Q).
\end{align*}
The last condition is excluded by assumption \eqref{eq:NR}.  The relative matrix on
the first line lies in $\SL(d,\mathbb R)$, so
\eqref{eq:commensurator} shows that it is not in $\mathscr S_d$.  By
\cite[Lemma~8]{MarklofStrom2014}, this is equivalent to
incommensurability of $\Gamma_u$ and $\Gamma_v$.
\end{proof}

We next justify the use of ``generic'' for an $\mathscr L$-nonresonant polyhedron. An \emph{open realisation stratum} is
a nonempty open subset $\mathcal U\subset(\mathbb{RP}^{d-1})^s$ such that,
for every $([n_1],\ldots,[n_s])\in\mathcal U$, unit representatives $n_\ell$ and support numbers $h_\ell$ can be chosen so that
\begin{equation*}
 \bigcap_{\ell=1}^s\{x\in\mathbb R^d:x\cdot n_\ell<h_\ell\}
\end{equation*}
is a bounded full-dimensional polyhedron of the fixed labelled combinatorial
type.

\begin{lemma}\label{lem:word-map}
For every nonempty reduced word $w\in\mathcal W_s$, the real-algebraic map
\begin{equation}\label{eq:projective-word-map}
 \mathcal R_w:(\mathbb{RP}^{d-1})^s\longrightarrow\PGL(d,\mathbb R),
 \qquad ([n_1],\ldots,[n_s])\longmapsto[R_w]
\end{equation}
is nonconstant on every nonempty open subset.
\end{lemma}

\begin{proof}
Write $R([n])$ for reflection in the hyperplane orthogonal to $[n]$.  For a
nonzero row vector $n$,
\begin{equation*} [R([n])]=[(n\,{}^tn)I-2\,{}^tn\,n].
\end{equation*}
Hence \eqref{eq:projective-word-map} is real algebraic.  If it were constant
on an open set, the analytic identity principle would make it constant on
$(\mathbb {RP}^{d-1})^s$; write its value as $[C]$.  Simultaneous rotation gives
\begin{equation*}
 [Q^{-1}CQ]=[C]\qquad(Q\in\SO(d)).
\end{equation*}
Thus $Q^{-1}CQ=c(Q)C$ for a continuous character
$c:\SO(d)\to\mathbb R^\times$.  The Frobenius norm gives
$c(Q)\in\{-1,1\}$, and connectedness of $\SO(d)$ gives $c\equiv1$.
The commutant of the standard $\SO(d)$-representation is scalar for
$d\geq3$, so the constant value would be $[I]$.

For $d=3$, Dekker's theorem
\cite[Theorem~2]{Dekker1959}, applied to $s$ cyclic groups of order two, gives
rotations $H_1,\ldots,H_s\in\SO(3)$, each of order two, for which
$a_\ell\mapsto H_\ell$ extends to an injective homomorphism
$\mathcal W_s\to\SO(3)$.  Since $S_\ell=-H_\ell$ is a reflection, for
$w=\ell_1\cdots\ell_r$ put $H_w=H_{\ell_1}\cdots H_{\ell_r}$ and
$S_w=S_{\ell_1}\cdots S_{\ell_r}$.  Then
\begin{equation*}
 S_w=(-I_3)^rH_w.
\end{equation*}
If $S_w$ were scalar, then $H_w=I_3$, contradicting injectivity.

For $d>3$, put
\begin{equation*}
 S_\ell^{(d)}=\operatorname{diag}(S_\ell,I_{d-3}).
\end{equation*}
The corresponding word product is
$S_w^{(d)}=S_{\ell_1}^{(d)}\cdots S_{\ell_r}^{(d)}
=\operatorname{diag}(S_w,I_{d-3})$, which cannot be scalar.  Thus
$\mathcal R_w$ is not identically $[I]$, a contradiction.
\end{proof}

The following proposition shows that an $\mathscr L$-nonresonant polyhedron is generic in both a measure-theoretic and topological sense.

\begin{proposition}\label{prop:genericity}
Fix $d\geq3$ and $\mathscr L=\mathbb Z^dM_0$.  In every open realisation stratum
$\mathcal U$ as above, the $\mathscr L$-nonresonant normal configurations
form a conull dense $G_\delta$ subset of $\mathcal U$.
\end{proposition}

\begin{proof}
The exceptional set is
\begin{equation*} \bigcup_{w\in\mathcal W_s\setminus\{e\}}\ \bigcup_{C\in\PGL(d,\mathbb Q)}
 \left\{\boldsymbol n\in\mathcal U:
 [M_0R_w(\boldsymbol n)M_0^{-1}]=C\right\}.
\end{equation*}
Both unions are countable.  By Lemma~\ref{lem:word-map}, each fibre is a
proper real-algebraic subset of $\mathcal U$, hence closed, null and nowhere
dense.  The complement is therefore conull and a dense $G_\delta$.
Adding support numbers and normalising the volume do not
change the conclusion.
\end{proof}

\begin{remark}\label{rem:vorobets}
For $d=3$, Vorobets \cite[Theorem~4]{Vorobets2005} proves generic faithfulness
of the reflection representation.  Together with the analyticity and
$\SO(3)$-equivariance argument above, this gives another proof of
Lemma~\ref{lem:word-map}.
\end{remark}

\begin{remark}\label{rem:d2}
In $d=2$, every orientation-reversing orthogonal matrix squares to the
identity.  If a polygon has parallel faces,
their two face reflections already have product $I$.  Otherwise choose three
distinct face labels $\ell_1,\ell_2,\ell_3$; then
$(R_{\ell_1}R_{\ell_2}R_{\ell_3})^2=I$, while
$\ell_1\ell_2\ell_3\ell_1\ell_2\ell_3$ is a
nonempty reduced word.  This shows that our nonresonance condition (which we require to hold for all reduced words) has no two-dimensional analogue.
\end{remark}

\section{Product equidistribution and higher-order collision laws}
\label{sec:collision-limit}

The diagonal flow compatible with the Boltzmann--Grad scaling is
\begin{equation*} A^t=\operatorname{diag}(e^{-(d-1)t},e^t,\ldots,e^t),
 \qquad D(\rho)=A^{\log(1/\rho)}.
\end{equation*}
We use the following velocity-dependent form of the product theorem from
\cite[Theorems~7 and 10]{MarklofStrom2014}.  Given translations
$\alpha_0,\ldots,\alpha_N$, we use the shorthand
$(Y_j,\eta_j)=(X_{\alpha_j},\mu_{\alpha_j})$, with
$(X_\alpha,\mu_\alpha)$ defined in Section~\ref{sec:phi}.  Furthermore, put
\begin{equation*} \psi_\alpha(M)=
 \begin{cases}
  \ASL(d,\mathbb Z)(I,\alpha)(M,0),&\alpha\notin\mathbb Q^d,\\
  \Gamma(q_\alpha)M,&\alpha\in\mathbb Q^d.
 \end{cases}
\end{equation*}

Matrices $M_0,\ldots,M_N\in\SL(d,\mathbb R)$ are called \emph{pairwise
incommensurable} when the arithmetic subgroups
$M_j^{-1}\SL(d,\mathbb Z)M_j$ are pairwise incommensurable in the sense of
\eqref{eq:commensurator-definition}.

\begin{theorem}\label{thm:product-equi}
Let $N\geq0$ and let $M_0,\ldots,M_N\in\SL(d,\mathbb R)$ be pairwise
incommensurable.  If
$\lambda$ is a Borel probability measure absolutely continuous with respect
to $\omega$, and
$F:S_1^{d-1}\times Y_0\times\cdots\times Y_N\to\mathbb R$ is bounded and
continuous, then
\begin{align*} &\lim_{t\to\infty}\int_{S_1^{d-1}}
 F\bigl(v,\psi_{\alpha_0}(M_0K(v)A^t),\ldots,
 \psi_{\alpha_N}(M_NK(v)A^t)\bigr)\,d\lambda(v)\\
 &\qquad=\int_{S_1^{d-1}}\int_{Y_0\times\cdots\times Y_N}
 F(v,g_0,\ldots,g_N)\,d\eta_0(g_0)\cdots d\eta_N(g_N)\,d\lambda(v).
\end{align*}
\end{theorem}

\begin{proof}
For $F$ independent of $v$, this is \cite[Theorems~7 and 10]
{MarklofStrom2014}.  The stated form follows by a standard approximation argument; see, e.g., \cite[Theorem~5.3]{MarklofStrom2010}.
\end{proof}

\begin{remark}\label{rem:nonproduct}
Without the incommensurability condition, Shah's theorem \cite[Theorem~1.4]{Shah1996} still gives an analogous convergence statement, but the limiting measure is generally not the product measure displayed in Theorem~\ref{thm:product-equi}.
\end{remark}

Fix 
$\boldsymbol\ell=(\ell_1,\ldots,\ell_n)$ and put, for $0\leq j\leq n-1$,
\begin{equation*} P_j=R_{\ell_1}\cdots R_{\ell_j},
 \qquad v_j=v_0P_j,
 \qquad K_j(v_0)=P_j^{-1}K(v_0).
\end{equation*}
Then $v_jK_j(v_0)=e_1$.  The first transformed point set is
$(\mathscr L-q_\star)K(v_0)D(\rho)$; after the $j$th collision,
$\mathscr L-m_j=\mathscr L$, and the next centres form
$\mathscr LP_j^{-1}K(v_0)D(\rho)$. An $\SL(d,\mathbb R)$-representative of this point set is
$\widetilde M_jK(v_0)A^t$, where
\begin{align*}
 \widetilde M_0&=M_0,\\
 \widetilde M_j
 &=J^jM_0R_{\ell_j}\cdots R_{\ell_1}
 =J^jM_0P_j^{-1}\in\SL(d,\mathbb R)
 \qquad (j\geq1).
\end{align*}
Here $J$ ensures the determinant of $\widetilde M_j$ is one but does not change the point set.
The first factor has translation $\alpha$, and the later
factors have translation $0$.
Define
\begin{equation*} \widehat O_j(v_0)=K_j(v_0)^{-1}K(v_j)
 =\operatorname{diag}(1,O_j(v_0)),\qquad O_j(v_0)\in\mathrm O(E).
\end{equation*}
It fixes $e_1$, commutes with $A^t$, and converts the transported frame to
$K(v_j)$.  Thus $\widehat O_0=I$ and
\begin{equation*} \mathbb Z^d\widetilde M_jK(v_0)A^t\widehat O_j(v_0)
 =\mathscr L K(v_j)A^t,
 \qquad \det O_j=(-1)^j.
\end{equation*}
For $j\geq1$, Remark~\ref{rem:Phi-covariance} applies even when
$\det O_j=-1$.

We next define the collision-record space before constructing the recursive
first-point map.  For $n\geq1$, let $\Omega_n$ consist of the regular records
\begin{equation*} \varpi=(v_0,\xi_1,b_1,\ldots,\xi_n,b_n),
 \qquad \xi_j>0,\quad b_j\in B_{v_{j-1}},
\end{equation*}
with recursion
\begin{equation}\label{eq:Omega-recursion}
 \ell_j=\face_{v_{j-1}}^-(b_j),\qquad
 W_j=(\ell_j,b_j),\qquad
 v_j=\mathcal R(v_{j-1},b_j),\qquad
 z_j=\mathcal Z(v_{j-1},b_j).
\end{equation}
Give $\Omega_n$ the disjoint-union topology over its face words, and put
\begin{equation*} d\sigma_\lambda^{(n)}(\varpi)
 =d\lambda(v_0)\prod_{j=1}^n d\xi_j\,db_j.
\end{equation*}
A face word $\boldsymbol\ell=(\ell_1,\ldots,\ell_n)$ is \emph{reduced} if
$\ell_{j+1}\ne\ell_j$ for $1\leq j<n$. Let $\Omega_n(\boldsymbol\ell)$ be the word
component consisting of the records in $\Omega_n$ with face word
$\boldsymbol\ell$. Every record in $\Omega_n$ has a reduced face word:
after a collision on face $\ell_j$ one has $v_j\cdot n_{\ell_j}>0$, so the
next incoming face cannot again be $\ell_j$. For a reduced face word
$\boldsymbol\ell$, let
\begin{equation*}
 Y_{\boldsymbol\ell}=S_1^{d-1}\times X_\alpha\times X_0^{n-1}
\end{equation*}
and write its elements as
$\mathbf y=(v_0,g_0,\ldots,g_{n-1})$. 
We construct a subset $D_{\boldsymbol\ell}\subset Y_{\boldsymbol\ell}$ recursively as follows.
Put $\alpha_0=\alpha$ and $\alpha_j=0$ for $j\geq1$, and define
\begin{equation}\label{eq:punctured-lattice}
 \Lambda_{\alpha_j}^{\circ}(g)=
 \begin{cases}
  \Lambda_0(g)\setminus\{0\},&\alpha_j=0,\\
  \Lambda_{\alpha_j}(g),&\alpha_j\ne0.
 \end{cases}
\end{equation}
Starting from $z_0=z_0(v_0)$, for
$j=0,\ldots,n-1$ consider the set
\begin{equation*} \bigl(\Lambda_{\alpha_j}^{\circ}(g_j)\widehat O_j(v_0)\bigr)\cap
 \bigl(\mathfrak Z_{v_j}(0,\infty)+z_j\bigr).
\end{equation*}
If this set has no unique point with least positive first coordinate, then we cannot complete the construction to obtain a point in $D_{\boldsymbol\ell}$. If it does, denote that point by $y_{j+1}$ and set
\begin{equation*} \xi_{j+1}=y_{j+1}\cdot e_1,
 \qquad b_{j+1}=\prperp(y_{j+1})-z_j.
\end{equation*}
If a grazing velocity, a face-cell boundary point, or a face different from $\ell_{j+1}$ occurs, the construction again fails. If none of the above occurs, put
$z_{j+1}=\mathcal Z(v_j,b_{j+1})$ and continue. Let
$D_{\boldsymbol\ell}\subset Y_{\boldsymbol\ell}$
be the set of inputs for which this construction succeeds at all $n$ stages, and let
\begin{equation*}
 \Theta_{\boldsymbol\ell}:D_{\boldsymbol\ell}
 \longrightarrow\Omega_n(\boldsymbol\ell)
\end{equation*}
be the resulting record map. For
$G\in C_b(\Omega_n(\boldsymbol\ell))$, define on the full input space
\begin{equation}\label{eq:word-observable}
 H_{\boldsymbol\ell}^{G}(\mathbf y)=
 \begin{cases}
  G(\Theta_{\boldsymbol\ell}(\mathbf y)),&\mathbf y\in D_{\boldsymbol\ell},\\
  0,&\mathbf y\notin D_{\boldsymbol\ell}.
 \end{cases}
\end{equation}

\begin{lemma}
\label{lem:record-map}
The set $D_{\boldsymbol\ell}$ and the map
$\Theta_{\boldsymbol\ell}$ are Borel. For every probability measure
$\lambda$ absolutely continuous with respect to $\omega$ and every
$G\in C_b(\Omega_n(\boldsymbol\ell))$, the function
$H_{\boldsymbol\ell}^{G}$ is continuous outside a set $E_{\boldsymbol\ell}$ of
$\nu$-measure zero, where $\nu=\lambda\otimes\mu_\alpha\otimes \mu_0^{\otimes (n-1)}$.
\end{lemma}

\begin{proof}
The exceptional configurations are those with a nonzero
point on a tube boundary, two points at the least longitudinal coordinate, a
terminal point on a face-cell boundary, a grazing velocity, no finite first
coordinate, or $v_j=v_{\rm s}$ at some stage. Tube-boundary points are null
by Siegel's first-moment formula, ties by the two-point estimate in the proof
of Lemma~\ref{lem:general-shadow}, face-cell boundaries and grazing
velocities by the flux formula \eqref{eq:flux-in}, and the absence of a finite
first point by Lemma~\ref{lem:general-shadow}(i). The chart-singular inputs
are null because $\lambda$ is absolutely continuous. The distinguished zero
vector has already been deleted in \eqref{eq:punctured-lattice}. The defining
incidence, avoidance, and cell-membership conditions are Borel. Away from the exceptional
configurations, local finiteness and the stability estimate
\eqref{eq:first-point-stability} show inductively that success or failure of
the prescribed word is locally constant and that, on
$D_{\boldsymbol\ell}$, the extracted point and all subsequent collision data
vary continuously. This proves the asserted continuity of
$H_{\boldsymbol\ell}^{G}$.
\end{proof}

\begin{proposition}
\label{prop:record-map-law}
For every bounded Borel function
$G:\Omega_n(\boldsymbol\ell)\to\mathbb R$,
\begin{align}
 &\int_{D_{\boldsymbol\ell}}
 G(\Theta_{\boldsymbol\ell}(v_0,g_0,\ldots,g_{n-1}))
 \,d\lambda(v_0)\prod_{j=0}^{n-1}d\mu_{\alpha_j}(g_j)\notag\\
 &\quad=\int_{\Omega_n(\boldsymbol\ell)}G(\varpi)
 \Phi_{\alpha,v_0}(\xi_1,b_1,z_0(v_0))
 \prod_{j=1}^{n-1}\Phi_{0,v_j}(\xi_{j+1},b_{j+1},z_j)
 \,d\sigma_\lambda^{(n)}(\varpi).
 \label{eq:record-map-law}
\end{align}
\end{proposition}

\begin{proof}
Haar invariance and \eqref{eq:Phi-covariance} reduce each factor to
Lemma~\ref{lem:general-shadow}, whose conditional density is
\begin{equation*}
 \Phi_{\alpha_j,v_j}(\xi_{j+1},b_{j+1},z_j)\,
 d\xi_{j+1}\,db_{j+1}.
\end{equation*}
Successive conditioning and \eqref{eq:positive-bin-disintegration} give
\eqref{eq:record-map-law}; the bounded Borel case follows by a monotone-class
argument.
\end{proof}

Define
\begin{equation}\label{eq:F-n}
 \mathcal F_{\alpha,\beta}^{(n)}(\varpi)=
 \Phi_{\alpha,v_0}(\xi_1,b_1,z_0(v_0))
 \prod_{j=1}^{n-1}
 \Phi_{0,v_j}(\xi_{j+1},b_{j+1},z_j).
\end{equation}
By \eqref{eq:Phi-normalisation}, these are probability densities, and they
are consistent since
\begin{equation*} \int_0^\infty\int_{B_{v_n}}
 \mathcal F_{\alpha,\beta}^{(n+1)}\,db_{n+1}\,d\xi_{n+1}
 =\mathcal F_{\alpha,\beta}^{(n)}.
\end{equation*}
The Ionescu--Tulcea theorem therefore yields the corresponding law on infinite records.

Now put
$t=\log(1/\rho)$ and
\begin{equation*}
 \begin{aligned}
 g_{j,\rho}(v_0)
 &=\psi_{\alpha_j}\bigl(\widetilde M_jK(v_0)A^t\bigr),\\
 \iota_{\boldsymbol\ell,\rho}(v_0)
 &=\bigl(v_0,g_{0,\rho}(v_0),\ldots,g_{n-1,\rho}(v_0)\bigr)
 \in Y_{\boldsymbol\ell}.
 \end{aligned}
\end{equation*}

We construct $D_{\boldsymbol\ell,\rho}\subset Y_{\boldsymbol\ell}$ recursively in the same way as $D_{\boldsymbol\ell}$ above. We now use the point sets
$\Lambda_{\alpha_j}^{\circ}(g_j)\widehat O_j(v_0)$ and set
$x_0=\beta(v_0)$. For $j=0,\ldots,n-1$, put $x_j=x^-_{v_{j-1}}(b_j)$ when $j\geq1$, and in factor $j$ use
the collision region
\begin{align}
 \mathfrak C_{\rho,j}(\xi,z_j)
 &=\bigl\{[s+\rho^d(x_j-p)\cdot v_j]e_1+z_j
 -\prperp(pK(v_j)):\notag\\[-1mm]
 &\hspace{42mm}0<s<\xi,\ p\in\mathcal P\bigr\}.
 \label{eq:finite-rho-collision-stage}
\end{align}
The first physical contact time is
\begin{equation*}
 \xi_{j+1}=\inf\bigl\{\xi>0:
 (\Lambda_{\alpha_j}^{\circ}(g_j)\widehat O_j(v_0))
 \cap\mathfrak C_{\rho,j}(\xi,z_j)\ne\varnothing\bigr\}.
\end{equation*}
If a unique centre $y_{j+1}$ belongs to
$\mathfrak C_{\rho,j}(\eta,z_j)$ for every $\eta>\xi_{j+1}$ sufficiently
close to $\xi_{j+1}$ and gives a regular collision on the prescribed
face, put
\begin{equation*}
 b_{j+1}=\prperp(y_{j+1})-z_j,\qquad
 x_{j+1}=x^-_{v_j}(b_{j+1}),\qquad
 z_{j+1}=\mathcal Z(v_j,b_{j+1}),
\end{equation*}
and continue. The inputs for which all $n$
stages succeed form $D_{\boldsymbol\ell,\rho}$. This set and the resulting
map are Borel by the same incidence-and-avoidance argument as in
Lemma~\ref{lem:record-map}.
Let
\begin{equation*}
 \Theta_{\boldsymbol\ell,\rho}:D_{\boldsymbol\ell,\rho}
 \longrightarrow\Omega_n(\boldsymbol\ell)
\end{equation*}
be the resulting record map.

For fixed $n$, let $\mathcal A_{n,\rho}$ be the event that the
first $n$ collisions exist and are regular. On this event define the physical
record
\begin{equation*}
 \varpi_n^{(\rho)}=
 \bigl(v_0,\xi_1^{(\rho)},b_1^{(\rho)},\ldots,
       \xi_n^{(\rho)},b_n^{(\rho)}\bigr)\in\Omega_n.
\end{equation*}
The following theorem identifies the limit distribution of
this record; taking $G\equiv1$ also implies that
$\lambda(\mathcal A_{n,\rho}^{c})\to0$.
 
\begin{theorem}\label{thm:impact-micro}
Let $d\geq3$, let $\mathcal P$ be $\mathscr L$-nonresonant,
and let $q_{\rho,\beta}(v)$ and $\lambda$ be as in
Theorem~\ref{thm:one-flight}.
Then, for every
$G\in C_b(\Omega_n)$,
\begin{equation}\label{eq:impact-convergence}
 \lim_{\rho\to0}\int_{\mathcal A_{n,\rho}}
 G\bigl(\varpi_n^{(\rho)}\bigr)\,d\lambda(v_0)
 =\int_{\Omega_n}G(\varpi)\,
 \mathcal F_{\alpha,\beta}^{(n)}(\varpi)\,
 d\sigma_\lambda^{(n)}(\varpi).
\end{equation}
\end{theorem}

\begin{proof}
Fix a reduced face word
$\boldsymbol\ell=(\ell_1,\ldots,\ell_n)$, which records the faces in
chronological order. Set $u_0=e$; after $j$ collisions, the corresponding
unfolded group word is $u_j=\ell_j\cdots\ell_1$ for $1\leq j\leq n-1$.
Since $\boldsymbol\ell$ is reduced, the words $u_0,\ldots,u_{n-1}$ are
distinct, and $u_j$ is reduced for every $1\leq j\leq n-1$. Since
$\widetilde M_j=M_{u_j}^{\circ}$, Lemma~\ref{lem:pairwise} says that the
$\widetilde M_j$ are pairwise incommensurable. Let
$\nu=\lambda\otimes\mu_\alpha\otimes\mu_0^{\otimes(n-1)}$ and
$\nu_{\boldsymbol\ell,\rho}
=(\iota_{\boldsymbol\ell,\rho})_*\lambda$.
By Theorem~\ref{thm:product-equi}, we have the weak convergence
\begin{equation}\label{eq:product-orbit-weak}
 \nu_{\boldsymbol\ell,\rho}\weakto\nu
 \qquad\text{on }Y_{\boldsymbol\ell}.
\end{equation}
For $j=0$, \eqref{eq:finite-rho-collision-stage} is
\eqref{eq:finite-rho-collision-set}.  For $j\geq1$, we have
\begin{equation*} (m_{j+1}-m_j)K(v_j)D(\rho)
 =\bigl[\xi_{j+1}+\rho^d(x_j-x_{j+1})\cdot v_j\bigr]e_1
  +b_{j+1}+z_j.
\end{equation*}
Moreover,
$\iota_{\boldsymbol\ell,\rho}(v_0)\in D_{\boldsymbol\ell,\rho}$ precisely when the first $n$
collisions exist, are regular, and have face word $\boldsymbol\ell$; on this set $\Theta_{\boldsymbol\ell,\rho}
(\iota_{\boldsymbol\ell,\rho}(v_0))$ is the physical record for each $\rho>0$.

For $G\in C_b(\Omega_n(\boldsymbol\ell))$, define
\begin{equation*}
 H_{\boldsymbol\ell,\rho}^{G}(\mathbf y)=
 \begin{cases}
  G(\Theta_{\boldsymbol\ell,\rho}(\mathbf y)),
      &\mathbf y\in D_{\boldsymbol\ell,\rho},\\
  0,&\mathbf y\notin D_{\boldsymbol\ell,\rho}.
 \end{cases}
\end{equation*}
The tube sandwich \eqref{eq:tube-sandwich}, first-point stability
\eqref{eq:first-point-stability}, and the $O(\rho^d)$ correction in
\eqref{eq:finite-rho-collision-stage} show inductively that, if
$\rho_k\to0$, $\mathbf y_k\to\mathbf y$, and
$\mathbf y\notin E_{\boldsymbol\ell}$ (where $E_{\boldsymbol\ell}$ is the $\nu$-null set defined in Lemma~\ref{lem:record-map}), then $\mathbf 1_{D_{\boldsymbol\ell,\rho_k}}(\mathbf y_k)
 =\mathbf 1_{D_{\boldsymbol\ell}}(\mathbf y)$ for all sufficiently large $k$, and, when $\mathbf y\in D_{\boldsymbol\ell}$,
\begin{equation*}
 \Theta_{\boldsymbol\ell,\rho_k}(\mathbf y_k) \longrightarrow\Theta_{\boldsymbol\ell}(\mathbf y).
\end{equation*}
In particular, this implies
\begin{equation}\label{eq:varying-record-map}
 H_{\boldsymbol\ell,\rho_k}^{G}(\mathbf y_k)
 \longrightarrow H_{\boldsymbol\ell}^{G}(\mathbf y).
\end{equation}

Since $\lVert H_{\boldsymbol\ell,\rho}^{G}\rVert_\infty\leq\lVert G\rVert_\infty$,
the extended continuous mapping theorem for varying
functions, applied to \eqref{eq:product-orbit-weak} and
\eqref{eq:varying-record-map}, now gives
\begin{align}
 \lim_{\rho\to0}
 \int H_{\boldsymbol\ell,\rho}^{G}
       (\iota_{\boldsymbol\ell,\rho}(v_0))\,d\lambda(v_0)\notag
 &=
 \int_{Y_{\boldsymbol\ell}}
 H_{\boldsymbol\ell}^{G}(\mathbf y)\,d\nu(\mathbf y)\notag\\
 &=
 \int_{\Omega_n(\boldsymbol\ell)}
 G(\varpi)\mathcal F_{\alpha,\beta}^{(n)}(\varpi)
 \,d\sigma_\lambda^{(n)}(\varpi),
 \label{eq:fixed-word-limit}
\end{align}
where the last identity is Proposition~\ref{prop:record-map-law}.

First take $G\equiv1$ in \eqref{eq:fixed-word-limit} and sum
over the finitely many reduced face words.
The successful finite-$\rho$ events associated with these
words partition $\mathcal A_{n,\rho}$, while
\eqref{eq:Phi-normalisation} makes the limiting word probabilities sum to
one. Hence
$\lambda(\mathcal A_{n,\rho}^{c})\to0$.
For general $G\in C_b(\Omega_n)$, apply
\eqref{eq:fixed-word-limit} to the restriction of $G$ to
each word component $\Omega_n(\boldsymbol\ell)$ and sum over the finitely
many reduced face words. This gives
\eqref{eq:impact-convergence}.
\end{proof}

\begin{remark}
In principle, it should be possible to prove an analogue of Theorem~\ref{thm:impact-micro} for polyhedra that are not $\mathscr L$-nonresonant, in any dimension $d\geq2$. As noted in Remark~\ref{rem:nonproduct}, the limiting measure would then no longer be a product measure, leading to persistent memory effects in the limiting random flight process. In the related setting of periodic slit barriers, this was carried out in \cite[Theorems~3(1) and~9]{BachurinKhaninMarklofPlakhov2011}.
\end{remark}

\section{Velocity and segment laws}
\label{sec:velocity}

For $z\in E$, define
\begin{equation*} a_{\alpha,\ell}(v,z;\xi)=
 \int_{B_{v,\ell}}\Phi_{\alpha,v}(\xi,b,z)\,db.
\end{equation*}
In view of \eqref{eq:flux-in},
\begin{equation*} a_{\alpha,\ell}(v,z;\xi)=
 \begin{cases}
 \displaystyle\int_{F_\ell}
 \Phi_{\alpha,v}\bigl(\xi,-\prperp(xK(v)),z\bigr)
 (-v\cdot n_\ell)\,dA_\ell(x),&v\cdot n_\ell<0,\\[2ex]
 0,&v\cdot n_\ell\geq0.
 \end{cases}
\end{equation*}
The one-flight velocity kernel is the probability kernel
\begin{equation}\label{eq:atomic-kernel}
 \mathsf p_{\alpha,z}(v;d\xi,du)=
 d\xi\sum_{\ell:v\cdot n_\ell<0}
 a_{\alpha,\ell}(v,z;\xi)\,\delta_{vR_\ell}(du).
\end{equation}
In distribution notation,
\begin{equation*} p_{\alpha,z}(v,\xi,u)=
 \sum_{\ell:v\cdot n_\ell<0}
 a_{\alpha,\ell}(v,z;\xi)\,\delta_{vR_\ell}(u).
\end{equation*}
Write $(a)_+=\max\{a,0\}$ and put
$|F_\ell|=\int_{F_\ell}dA_\ell$.  The geometric differential scattering cross
section is the finite measure
\begin{equation*} \Sigma_v(du)=
 \sum_{\ell=1}^s|F_\ell|(-v\cdot n_\ell)_+
 \delta_{vR_\ell}(du).
\end{equation*}
Boundary strata and grazing faces are flux-null.  

Let
\begin{equation*} \begin{aligned}
& \mathcal V_n :\Omega_n\longrightarrow
 S_1^{d-1}\times(\mathbb R_{>0}\times S_1^{d-1})^n,\\
 & \mathcal V_n(\varpi) =(v_0,\xi_1,v_1,\ldots,\xi_n,v_n),
 \end{aligned}
\end{equation*}
where the velocities are generated by \eqref{eq:Omega-recursion}.
The flight-length/velocity limit is the Borel probability
measure
\begin{equation}\label{eq:velocity-pushforward}
\widehat{\mathbb P}_{\alpha,\beta}^{(n)}
 =(\mathcal V_n)_*
 \bigl(\mathcal F_{\alpha,\beta}^{(n)}\sigma_\lambda^{(n)}\bigr).
\end{equation}
Define
\begin{equation*} \mathcal S_n:\Omega_n\longrightarrow
 (\mathbb R^d\setminus\{0\})^n,\qquad
 \mathcal S_n(\varpi)=(\xi_1v_0,\xi_2v_1,\ldots,\xi_nv_{n-1}),
\end{equation*}
and the limiting segment law is
\begin{equation}\label{eq:segment-pushforward}
 \mathbb P_{\alpha,\beta}^{(n)}
 =(\mathcal S_n)_*
 \bigl(\mathcal F_{\alpha,\beta}^{(n)}\sigma_\lambda^{(n)}\bigr).
\end{equation}
Thus, for a bounded continuous test function $G$,
\begin{equation*} \int G\,d\mathbb P_{\alpha,\beta}^{(n)}
 =\int_{\Omega_n}
 G(\xi_1v_0,\ldots,\xi_nv_{n-1})
 \mathcal F_{\alpha,\beta}^{(n)}\,d\sigma_\lambda^{(n)}.
\end{equation*}

On $\mathcal A_{n,\rho}$, the images
$\mathcal V_n(\varpi_n^{(\rho)})$ and
$\mathcal S_n(\varpi_n^{(\rho)})$ are precisely the physical
flight-length/velocity and rescaled-segment tuples.

\begin{corollary}\label{cor:pushforward}
Under the assumptions of Theorem~\ref{thm:impact-micro}, for every
\begin{equation*}
 G\in C_b\!\left(S_1^{d-1}\times
 (\mathbb R_{>0}\times S_1^{d-1})^n\right),\qquad
 H\in C_b\!\left((\mathbb R^d\setminus\{0\})^n\right),
\end{equation*}
we have
\begin{align*}
 \lim_{\rho\to0}\int_{\mathcal A_{n,\rho}}
 G\bigl(\mathcal V_n(\varpi_n^{(\rho)})\bigr)\,d\lambda(v_0)
 &=\int G\,d\widehat{\mathbb P}_{\alpha,\beta}^{(n)},\\
 \lim_{\rho\to0}\int_{\mathcal A_{n,\rho}}
 H\bigl(\mathcal S_n(\varpi_n^{(\rho)})\bigr)\,d\lambda(v_0)
 &=\int H\,d\mathbb P_{\alpha,\beta}^{(n)}.
\end{align*}
\end{corollary}

\begin{proof}
The maps $\mathcal V_n$ and $\mathcal S_n$ are continuous
in the disjoint-word topology of $\Omega_n$. Hence
$G\circ\mathcal V_n$ and $H\circ\mathcal S_n$ belong to $C_b(\Omega_n)$, and the assertions
follow from Theorem~\ref{thm:impact-micro} and the definitions
\eqref{eq:velocity-pushforward} and \eqref{eq:segment-pushforward}.
\end{proof}

For a fixed face word and $n\geq2$, the segment map factors through the
embedding
\begin{equation*}
 (v_0,\xi_1,\ldots,\xi_n)\longmapsto
 (\xi_1v_0,\ldots,\xi_nv_{n-1})
\end{equation*}
of a $(d-1)+n$-dimensional manifold into an
$nd$-dimensional space: the
first segment recovers $(v_0,\xi_1)$, and the word determines the later
velocities.  Its image has codimension $(n-1)(d-1)$ and is Lebesgue null.
Since there are
finitely many words, \eqref{eq:segment-pushforward} is singular for $n\geq2$.

Since the next-flight kernel generally depends on the departure parameter
$z_j$ as well as on $v_j$, the velocity projection alone need not be Markov,
although \eqref{eq:velocity-pushforward} still determines all its
finite-dimensional distributions.

\section{Macroscopic limit and continuous-time flight}\label{sec:macro}

Put $\varepsilon=\rho^{d-1}$ and
\begin{equation*} \mathcal Y_\rho=T^1(\varepsilon\mathcal K_\rho),
\end{equation*}
where $T^1$ denotes the completed billiard phase space defined in
Section~\ref{sec:geometry}.  Away from its boundary it is
$(\varepsilon\mathcal K_\rho^{\circ})\times S_1^{d-1}$.
For $(Q,V)\in\mathcal Y_\rho$ whose microscopic trajectory
is defined up to time $\rho^{-(d-1)}t$, set
\begin{equation}\label{eq:macro-flow}
 F_{t,\rho}(Q,V)=
 \left(\rho^{d-1}q(\rho^{-(d-1)}t),
 v(\rho^{-(d-1)}t)\right).
\end{equation}
Here $(q(s),v(s))$ is the microscopic billiard trajectory
with $q(0)=\rho^{-(d-1)}Q$ and $v(0)=V$. Extend $F_{t,\rho}$ in any fixed
measurable fashion on the Liouville-null set where that trajectory has
terminated at a nonregular collision before this time.
Let
\begin{equation} \mu_{\rm in}(dQ,dV)=f_{\rm in}(Q,V)\,dQ\,d\omega(V)
\label{eq:muin}
\end{equation}
be an absolutely continuous probability measure.  Since the obstacle volume
fraction is $\rho^d\vol(\mathcal P)$, we have
$\mu_{\rm in}(\mathcal Y_\rho^c)\to0$.  Averaging over $Q$ puts the first
affine-lattice factor in the universal regime.  Define
the macroscopic impact-resolved measure on
$\mathbb R^d\times\Omega_n$ by the density below, writing $Q_0$ for the
initial macroscopic position:
\begin{equation}\label{eq:macro-density}
 f_{\rm in}(Q_0,v_0)\Phi_{v_0}(\xi_1,b_1)
 \prod_{j=1}^{n-1}
 \Phi_{0,v_j}(\xi_{j+1},b_{j+1},z_j)
\end{equation}
with respect to
\begin{equation*} dQ_0\,d\omega(v_0)\prod_{j=1}^n d\xi_j\,db_j.
\end{equation*}

For fixed $n$, let $\mathcal A_{n,\rho}^{\rm mac}\subset\mathcal Y_\rho$ be the set of initial
points for which the first $n$ collisions exist and are regular. On this set,
let $\mathcal R_{n,\rho}(Q,V)$ be the impact-resolved tuple
\begin{equation*}
 (V,\xi_1^{(\rho)},b_1^{(\rho)},\ldots,
 \xi_n^{(\rho)},b_n^{(\rho)})\in\Omega_n.
\end{equation*}
We will see below that $\mu_{\rm in}(\mathcal Y_\rho\setminus
 \mathcal A_{n,\rho}^{\rm mac})=0$ for every sufficiently small $\rho>0$.  
 
\begin{theorem}\label{thm:macro-impact}
Let $d\geq3$, let $\mathcal P$ be
$\mathscr L$-nonresonant, and let $\mu_{\rm in}$ be as in
\eqref{eq:muin}. Then, for every $G\in C_b(\mathbb R^d\times\Omega_n)$,
\begin{align}
 &\lim_{\rho\to0}\int_{\mathcal A_{n,\rho}^{\rm mac}}
 G\bigl(Q,\mathcal R_{n,\rho}(Q,V)\bigr)
 f_{\rm in}(Q,V)\,dQ\,d\omega(V)\notag\\
 &\quad=\int_{\mathbb R^d\times\Omega_n}G(Q_0,\varpi)
 f_{\rm in}(Q_0,v_0)\Phi_{v_0}(\xi_1,b_1) \\ 
 & \qquad \times
 \prod_{j=1}^{n-1}\Phi_{0,v_j}(\xi_{j+1},b_{j+1},z_j)
 \,dQ_0\,d\omega(v_0)\prod_{j=1}^n d\xi_j\,db_j.
 \label{eq:macro-test-limit}
\end{align}
\end{theorem}

Note that the limit is independent of $M_0$ except through the nonresonance hypothesis.

\begin{proof}
We use the fundamental-domain reduction of
\cite[Section~4.5]{MarklofStrom2011}.  Let $\mathscr F$ be a bounded
fundamental cell of $\mathscr L$ and write the microscopic initial point as $m+r$, with
$m\in\mathscr L$ and $r\in\mathscr F$.
First take $f_{\rm in}$ continuous and compactly supported
and insert an auxiliary compact cutoff in the full record variable
$\varpi$.  The spatial weight is
\begin{equation}\label{eq:macro-riemann-sum}
 \varepsilon^d\sum_{m\in\mathscr L}
 G\bigl(\varepsilon(m+r),\varpi\bigr)
 f_{\rm in}\bigl(\varepsilon(m+r),v_0\bigr).
\end{equation}
As $\varepsilon\to 0$, this Riemann sum converges uniformly in $r$, $v_0$, and $\varpi$ on compacta, to
\begin{equation*}
 \int_{\mathbb R^d}G(Q,\varpi)f_{\rm in}(Q,v_0)\,dQ.
\end{equation*}
The collision record depends on $(r,v_0)$ but not on $m$, and the excluded
part of $\mathscr F$ has volume $O(\rho^d)$.

At fixed $\rho$, grazing contacts and contacts
with lower-dimensional face strata form a null set for billiard flux
measure. Infinite flights are also flux-null: unless the velocity is
orthogonal to a nonzero vector of the dual lattice, its forward linear orbit
modulo $\mathscr L$ is dense and hence enters the open obstacle. The
exceptional velocities form a countable union of great subspheres. The
regular collision map preserves flux measure, while Liouville measure in
free-flight coordinates is flux measure times flight time. Finite iteration
therefore gives
$\mu_{\rm in}(\mathcal Y_\rho\setminus
\mathcal A_{n,\rho}^{\rm mac})=0$, since
$\mu_{\rm in}$ is absolutely continuous with respect to $dQ\,d\omega$.

For almost every $r$, $-rM_0^{-1}\notin\mathbb Q^d$.  Apply
Theorem~\ref{thm:impact-micro} with $\beta=0$ and dominated convergence in
$r$.  The first factor is $\Phi_{v_0}$ and the later, lattice-centred factors
are $\Phi_{0,v_j}$.  Equation~\eqref{eq:macro-riemann-sum} now gives
\eqref{eq:macro-test-limit}. Tightness and continuity-set
approximation remove the auxiliary record cutoff; $L^1$ approximation then
yields the claim for general $f_{\rm in}$.
\end{proof}

Let $T_0=0$ and $T_n=\xi_1+\cdots+\xi_n$.  For
$T_j\leq t<T_{j+1}$ define
\begin{equation*} \Xi(t)=
 \left(Q_0+\sum_{k=1}^j\xi_kv_{k-1}
 +(t-T_j)v_j,\ v_j\right).
\end{equation*}
The impact-extended process will be defined in the next section.

Let $\Prob_{\rm in}$ denote the path law determined by \eqref{eq:macro-density}. Since $0\leq\Phi_{\alpha,v}\leq1$, integration over the simplex
$\xi_1+\cdots+\xi_n\leq T$ gives the uniform collision-count bound
\begin{equation}\label{eq:factorial-bound}
 \Prob_{\rm in}(T_n\leq T)\leq\frac{(A_{\mathcal P}T)^n}{n!}.
\end{equation}
In particular, the limiting flight almost surely
has only finitely many collisions in every bounded time interval.

\begin{corollary}\label{cor:continuous-time}
Let $d\geq3$, let $\mathcal P$ be $\mathscr L$-nonresonant, and let
$\mu_{\rm in}$ be as in \eqref{eq:muin}.  For any $t_1,\ldots,t_M\in\mathbb R_{\geq0}$ and
any Borel sets $D_1,\ldots,D_M\subset\mathbb R^d\times S_1^{d-1}$ with
$(dQ\,d\omega)(\partial D_j)=0$, $j=1,\ldots,M$, we have
\begin{align*} &\lim_{\rho\to0}
 \mu_{\rm in}\Bigl\{(Q,V)\in\mathcal Y_\rho:
 F_{t_j,\rho}(Q,V)\in D_j,\ j=1,\ldots,M\Bigr\}\\
 &\qquad=
 \Prob_{\rm in}\bigl(\Xi(t_j)\in D_j,\ j=1,\ldots,M\bigr).
\end{align*}
The convergence is uniform for $(t_1,\ldots,t_M)$ in compact subsets of
$\mathbb R_{\geq0}^M$.
\end{corollary}

\begin{proof}
This is the counterpart of \cite[Theorem~1.4]{MarklofStrom2011}.  We follow
\cite[Sections~5.1, 5.2 and 5.4]{MarklofStrom2011}, with
Theorem~\ref{thm:macro-impact} in place of the segment-limit theorem used
there. Let $n_j$ be the number of collisions before time $t_j$ and put $N=\max_j n_j$; the first $N+1$ flights determine the observed states.  The hypersurfaces $\{T_k=t_j\}$,
the impact-cell boundaries, the grazing records, and the exceptional chart
direction are null.

For each fixed face word and fixed flight variables, the map from
$(Q_0,v_0)$ to an observed position--velocity pair has absolute Jacobian
one: the position is a translation of $Q_0$, while the observed velocity is
$v_0$ multiplied by a fixed orthogonal product of reflection matrices.
Thus every one-time marginal is absolutely continuous with respect to
$dQ\,d\omega$.

The collision-count truncation and time bracketing in \cite{MarklofStrom2011} now apply, with geometric factors bounded by $A_{\mathcal P}$.  For
each fixed $N$ and $T$, Theorem~\ref{thm:macro-impact}, applied to the
continuity event $T_N\leq T$, transfers the finite-$\rho$ collision-count tail
to the limiting one.  Equation~\eqref{eq:factorial-bound} then bounds this tail
uniformly for $T$ in compact intervals.  Letting $N\to\infty$, followed by the
standard tightness and $L^1$ approximations, proves the claim.
\end{proof}

\section{A continuous-time Markov process}\label{sec:kinetic}

Following \cite[Section~6]{MarklofStrom2011}, we
enlarge the state space by the remaining time and the data of the next collision:
\begin{equation*} \mathcal X_{\mathcal P}=
 \{(Q,V,\xi,W):Q\in\mathbb R^d,\ V\in S_1^{d-1},\
 \xi\geq0,\ W\in\mathcal B_V\}.
\end{equation*}
Here $W=(\ell,b)$ is the next labelled impact parameter.  
Write $\ell(W)=\ell$ and $b(W)=b$, and set
\begin{equation*} m_V(dW)=\sum_{\ell=1}^s
 \mathbf 1_{B_{V,\ell}}(b)\,\delta_\ell(d\ell)\,db,
 \qquad W=(\ell,b),
\end{equation*}
so that the reference measure is
\begin{equation*} d\sigma_1=dQ\,d\omega(V)\,d\xi\,m_V(dW).
\end{equation*}

Let $T_0=0$ and $T_n=\xi_1+\cdots+\xi_n$.  If
$T_j\leq t<T_{j+1}$, define
\begin{equation}\label{eq:extended-flight}
 \widehat\Xi(t)=
 \left(Q_0+\sum_{k=1}^j\xi_kv_{k-1}+(t-T_j)v_j,\,
 v_j,\,T_{j+1}-t,\,W_{j+1}\right).
\end{equation}
At the end of a flight with state $(V,W)$, $W=(\ell,b)$, put
\begin{equation*} V^+=\mathcal R(V,b)=VR_\ell,
 \qquad z^+=\mathcal Z(V,b).
\end{equation*}
The new velocity, remaining time and labelled impact parameter have transition
law
\begin{align}
 \mathsf J((V,W);dV'\,d\xi'\,dW') =\delta_{V^+}(dV')\,
 \Phi_{0,V^+}\bigl(\xi',b(W'),z^+\bigr)\,
 d\xi'\,m_{V^+}(dW').
 \label{eq:impact-transition}
\end{align}
The Borel property of the incidence and avoidance events in
\eqref{eq:Phi-definition}, together with disintegration, shows that
$\mathsf J$ is a Borel probability kernel.  Hence $\widehat\Xi$ is a
time-homogeneous Markov process.

More explicitly, from $x=(Q,V,\xi,W)$ the process follows
$(Q+tV,V,\xi-t,W)$ for $0\leq t<\xi$ and, at $t=\xi$, jumps at the
collision position $Q+\xi V$ to a state
$(Q+\xi V,V',\xi',W')$ distributed according to
$\mathsf J((V,W);dV'\,d\xi'\,dW')$.  If $\xi=0$, we keep the initial state at
time zero and apply the jump before every positive time.  Iteration defines
the process from every state.  The same simplex estimate as
in \eqref{eq:factorial-bound}, applied after the first jump and uniformly in
the initial state, shows that, almost surely, there are at most finitely many collisions in every bounded time window.

Let $\mathcal G_\rho\subset\mathcal Y_\rho$ be the full-measure
Borel set on which the next microscopic collision from
$(\rho^{-(d-1)}Q,V)$ exists and is regular.  If $\ell_1$ and
$x_1^{(\rho)}$ are the face label and local boundary point of that collision,
put
\begin{equation*}
 W_1^{(\rho)}(Q,V)
 =\left(\ell_1,-\prperp\bigl(x_1^{(\rho)}K(V)\bigr)\right)
 \in\mathcal B_V.
\end{equation*}
Define
\begin{equation}\label{eq:lift-map}
 \mathcal E_\rho(Q,V)=
 \left(Q,V,\rho^{d-1}
 \tau_1\bigl(\rho^{-(d-1)}Q,V;\rho\bigr),
 W_1^{(\rho)}(Q,V)\right)
\end{equation}
for $(Q,V)\in\mathcal G_\rho$ and extend to a measurable map to all of $\mathcal Y_\rho$.

Set
$\widehat F_{t,\rho}=\mathcal E_\rho\circ F_{t,\rho}$.
For the initial density in Section~\ref{sec:macro}, define
\begin{equation*} f_{\rm lift}(Q,V,\xi,(\ell,b))
 =f_{\rm in}(Q,V)\Phi_V(\xi,b),
 \qquad b\in B_{V,\ell}.
\end{equation*}
This is a probability density with respect to $\sigma_1$.  Let
$\Prob_{f_{\rm lift}}$ denote the law of \eqref{eq:extended-flight} with this
initial density.

\begin{theorem}\label{thm:extended-limit}
Let $d\geq3$, let $\mathcal P$ be $\mathscr L$-nonresonant, and let
$\mu_{\rm in}$ be as in \eqref{eq:muin}. For any $t_1,\ldots,t_M\in\mathbb R_{\geq0}$ and Borel sets
$\mathcal D_1,\ldots,\mathcal D_M\subset\mathcal X_{\mathcal P}$ with
$\sigma_1(\partial\mathcal D_j)=0$, $j=1,\ldots,M$,
\begin{align*} &\lim_{\rho\to0}
 \mu_{\rm in}\Bigl\{(Q,V)\in\mathcal Y_\rho:
 \widehat F_{t_j,\rho}(Q,V)\in\mathcal D_j,\
 j=1,\ldots,M\Bigr\}\\
 &\qquad=
 \Prob_{f_{\rm lift}}\bigl(\widehat\Xi(t_j)\in\mathcal D_j,\
 j=1,\ldots,M\bigr).
\end{align*}
The convergence is uniform for $(t_1,\ldots,t_M)$ in compact subsets of
$\mathbb R_{\geq0}^M$.
\end{theorem}

\begin{proof}
We follow \cite[Section~6.1]{MarklofStrom2011}, using
Theorem~\ref{thm:macro-impact} as the impact-resolved analogue of
\cite[Eq.~(6.12)]{MarklofStrom2011}.  For a fixed collision-count vector one
retains one further flight, since $(\xi,W)$ records the flight straddling the
last observation time.  Relative to
\cite[Theorem~6.1]{MarklofStrom2011}, its future-velocity coordinate is
replaced by $W$ and its transition density by
\eqref{eq:impact-transition}.

By a standard $L^1$ approximation argument, we may assume without loss of generality that $f_{\rm in}$ is
bounded and compactly supported.  Propagation is finite, the impact bundle is
bounded, the remaining flight is tight by
Lemma~\ref{lem:shadow-stability}(iii), and \eqref{eq:factorial-bound} controls
the collision count.

After summing over face words, the flux change of variables and
$\vol_E(B_V)\leq A_{\mathcal P}$ give the analogue of
\cite[Eqs.~(6.13)--(6.14)]{MarklofStrom2011}: for each bounded Borel test
set $\mathcal D$, the unweighted $n$-collision reference-measure contribution
is at most $(A_{\mathcal P}t)^n\sigma_1(\mathcal D)/n!$.  Hence, since
$f_{\rm lift}=f_{\rm in}\Phi_V$ and $0\leq\Phi_V\leq1$, the
$f_{\rm lift}$-weighted contribution is at most $\|f_{\rm in}\|_\infty (A_{\mathcal P}t)^n\sigma_1(\mathcal D)/n!$. This and \eqref{eq:factorial-bound} give the continuity-set approximation and
compact-time uniformity as in \cite[Section~6.1]{MarklofStrom2011}.
\end{proof}

Decomposition by collision number and the facewise map
\[
 (V,b)\longmapsto(VR_\ell,\mathcal Z(V,b)),
 \qquad b\in B_{V,\ell},
\]
which preserves $d\omega(V)\,db$ by
\eqref{eq:flux-in}--\eqref{eq:flux-out}, show that an initial density remains
absolutely continuous with respect to $\sigma_1$.
For $f\in L^1_{\rm loc}(\mathcal X_{\mathcal P},\sigma_1)$, define $K_tf$ by
\begin{equation} \int_{\mathcal D}K_tf\,d\sigma_1
 =\int_{\mathcal X_{\mathcal P}}
 \Prob_x\{\widehat\Xi(t)\in\mathcal D\}f(x)\,d\sigma_1(x)
 \label{eq:KProb}
\end{equation}
for every bounded Borel set $\mathcal D$, where $\Prob_x$ is the law started
at $x$; finite propagation makes the right-hand side locally finite.
Spatial translation covariance is understood also after taking $Q$ modulo
any full-rank lattice $\mathscr L_Q$.

\begin{proposition}\label{prop:stationary-density}
The locally integrable density
\begin{equation*} f_*(Q,V,\xi,(\ell,b))=\Phi_V(\xi,b)
\end{equation*}
satisfies $K_tf_*=f_*$ for all $t\geq0$.  On
$\mathbb R^d/\mathscr L_Q$ it becomes a probability density after
multiplication by $\vol(\mathbb R^d/\mathscr L_Q)^{-1}$.
\end{proposition}

\begin{proof}
It suffices to test the two locally finite measures against bounded Borel
continuity sets $\mathcal D$ whose $Q$-projection is contained in a bounded
set $A$.  Fix $t\geq0$, choose a bounded Borel set $C$ with
$A+\overline{B(0,t)}\subset C$, and put
\[
 h(Q,V)=c\,\mathbf 1_C(Q),\qquad
 c=[\vol(C)\,\omega(S_1^{d-1})]^{-1}.
\]
Thus $h\,dQ\,d\omega$ is a probability measure.  Every unit-speed
trajectory whose state at time $t$ belongs to $\mathcal D$ starts with its
position in $A+\overline{B(0,t)}$, where $h=c$; on the time-zero event
$\mathcal E_\rho^{-1}(\mathcal D)$ we likewise have $h=c$.  On the part of
the labelled collision section corresponding to $F_\ell$, the map
\[
 (V,b)\longmapsto(VR_\ell,\mathcal Z(V,b)),
 \qquad b\in B_{V,\ell},
\]
preserves $d\omega(V)\,db$ by \eqref{eq:flux-in}--\eqref{eq:flux-out}.
Liouville invariance of the microscopic billiard flow therefore gives, for
every $\rho$,
\[
 \int_{\mathcal Y_\rho}\!\mathbf 1_{\mathcal D}
   (\widehat F_{t,\rho}(Q,V))h(Q,V)\,dQ\,d\omega(V)
 =\int_{\mathcal Y_\rho}\!\mathbf 1_{\mathcal D}
   (\mathcal E_\rho(Q,V))h(Q,V)\,dQ\,d\omega(V).
\]
Apply Theorem~\ref{thm:extended-limit} at times $t$ and $0$.  Its initial
lift is $hf_*$, and finite propagation gives
$K_t(hf_*)=cK_tf_*$ on the $Q$-projection $A$, while $hf_*=cf_*$ there.
Consequently
\[
 \int_{\mathcal D}K_tf_*\,d\sigma_1
 =\int_{\mathcal D}f_*\,d\sigma_1,
\]
which yields
$K_tf_*=f_*$ in $L^1_{\rm loc}$.  Finally,
\eqref{eq:Phi-normalisation} gives the stated normalisation on the torus.
\end{proof}

\begin{proposition}\label{prop:semigroup}
The family $\{K_t:t\geq0\}$ forms a semigroup on
$L^1_{\rm loc}(\mathcal X_{\mathcal P},\sigma_1)$ and a positive,
mass-preserving contraction semigroup on
$L^1(\mathcal X_{\mathcal P},\sigma_1)$.
\end{proposition}

\begin{proof}
As in \cite[Eqs.~(6.23)--(6.27)]{MarklofStrom2011}, decompose
$K_t=\sum_{n\geq0}K_t^{(n)}$ according to the number of collisions before
time $t$.  Here the transition density is \eqref{eq:impact-transition}, and the bound $0\leq\Phi_{0,V}\leq1$ gives
\begin{equation*} \lVert K_t^{(n)}\rVert_{L^1\to L^1}
 \leq\frac{(A_{\mathcal P}t)^{n-1}}{(n-1)!},
 \qquad n\geq1.
\end{equation*}
The series therefore converges in operator norm on $L^1$, locally uniformly
in $t$.  Concatenation of the impact-resolved paths gives
$K_{t+s}=K_tK_s$. Positivity, mass preservation and the contraction property follow from \eqref{eq:KProb}, cf.~the remark after \cite[Eq.~(6.16)]{MarklofStrom2011}. The strong continuity required of a contraction semigroup follows from \cite[Eqs.~(6.32)--(6.34)]{MarklofStrom2011}.
\end{proof}

Let us now turn to the generalised Boltzmann equation, which is
the Kolmogorov forward equation for the Markovian random flight process. It
is the polyhedral impact-variable analogue of
\cite[Eq.~(2.30)]{CagliotiGolse2010} and
\cite[Section~6.3, especially Eq.~(6.37)]{MarklofStrom2011}.

For $V\cdot n_\ell>0$ and $x\in F_\ell$,
put
\begin{equation*} U_\ell(V)=VR_\ell,\qquad
 b_\ell^-(V,x)=-\prperp\bigl(xK(U_\ell(V))\bigr),\qquad
 z_\ell^+(V,x)=\prperp\bigl(xK(V)\bigr).
\end{equation*}
Thus $b_\ell^-(V,x)\in B_{U_\ell(V),\ell}$ and
$z_\ell^+(V,x)\in C_{V,\ell}$.  By
\eqref{eq:flux-in}--\eqref{eq:flux-out}, reflection preserves the facewise
flux measure:
\begin{equation}\label{eq:flux-reflection}
 d\omega(U_\ell(V))\,db_\ell^-(V,x)
 =d\omega(V)\,dz_\ell^+(V,x).
\end{equation}
Here \eqref{eq:flux-reflection} is an equality of pushforward measures: both
sides are the image of
\(
 \mathbf 1_{\{V\cdot n_\ell>0\}}(V\cdot n_\ell)
 \,d\omega(V)\,dA_\ell(x)
\)
under the corresponding coordinate map.

Let $f_0\in L^1(\sigma_1)$ be an initial density and put $f(t)=K_tf_0$.
Suppose that, locally for $t>0$, the weak derivatives $\partial_tf$,
$V\cdot\nabla_Qf$ and $\partial_\xi f$ are integrable and that $f$ has the
characteristic collision value
\begin{equation*}
 f(t,Q,V,0,W)
 :=\lim_{s\to0}f(t,Q-sV,V,s,W)
\end{equation*}
almost everywhere, with the collision integral in
\eqref{eq:kinetic-equation} locally integrable.  Then, for
$W=(\ell',b)\in\mathcal B_V$, $f$ satisfies in the distributional sense for
$t>0$ and $\xi>0$, with initial condition $f(0)=f_0$, 
\begin{equation}
   \bigl(\partial_t+V\cdot\nabla_Q-\partial_\xi\bigr)f(t,Q,V,\xi,W)
      =[\mathcal C_{\mathcal P}f](t,Q,V,\xi,W)
 \label{eq:kinetic-equation}
\end{equation}
with the collision operator
\begin{multline}
[\mathcal C_{\mathcal P}f](t,Q,V,\xi,(\ell',b)) =\sum_{\ell:V\cdot n_\ell>0}\int_{F_\ell}
 f\bigl(t,Q,U_\ell(V),0,(\ell,b_\ell^-(V,x))\bigr) \\ \times \Phi_{0,V}\bigl(\xi,b,z_\ell^+(V,x)\bigr)
 (V\cdot n_\ell)\,dA_\ell(x).
 \label{eq:CP}
\end{multline}

This follows from the analogous calculation in
\cite[Section~6.3]{MarklofStrom2011}.  Transport between collisions gives the
left-hand side.  At $\xi=0$, \eqref{eq:flux-reflection} and the conditional
density $\Phi_{0,V}(\xi,b,z_\ell^+)\,d\xi\,db$ give the right-hand side.  The
new features are the face-label sum, the impact variable, and the flux factor
$V\cdot n_\ell$.  No differentiability of $\Phi_{0,V}$ is required, and no
separate loss term appears because $-\partial_\xi f$ accounts for passage
through $\xi=0$.

For the physical initial density $f_{\rm in}$,
Theorem~\ref{thm:extended-limit} identifies the limiting extended density as
$K_tf_{\rm lift}$.  Its position--velocity marginal is
\begin{equation*} [L_tf_{\rm in}](Q,V)
 =\int_0^\infty\int_{\mathcal B_V}
 [K_tf_{\rm lift}](Q,V,\xi,W)\,m_V(dW)\,d\xi.
\end{equation*}
Note that, as in the case of the periodic Lorentz gas, the family $\{L_t:t\geq0\}$ does not in general form a
semigroup.


\begin{thebibliography}{BKMP11}

\bibitem[BKMP11]{BachurinKhaninMarklofPlakhov2011}
P.~Bachurin, K.~Khanin, J.~Marklof, and A.~Plakhov.
\newblock Perfect retroreflectors and billiard dynamics.
\newblock {\em Journal of Modern Dynamics}, 5(1):33--48, 2011.
\newblock \url{https://doi.org/10.3934/jmd.2011.5.33}.

\bibitem[CG10]{CagliotiGolse2010}
E.~Caglioti and F.~Golse.
\newblock On the {B}oltzmann--{G}rad limit for the two dimensional periodic
  {L}orentz gas.
\newblock {\em Journal of Statistical Physics}, 141(2):264--317, 2010.
\newblock \url{https://doi.org/10.1007/s10955-010-0046-1}.

\bibitem[Dek59]{Dekker1959}
T.~J. Dekker.
\newblock On free products of cyclic rotation groups.
\newblock {\em Canadian Journal of Mathematics}, 11:67--69, 1959.
\newblock \url{https://doi.org/10.4153/CJM-1959-009-7}.

\bibitem[EE11]{EE1911}
P.~Ehrenfest and T.~Ehrenfest.
\newblock Begriffliche Grundlagen der statistischen Auffassung in der
  Mechanik.
\newblock In F.~Klein and C.~M{\"u}ller, editors, {\em Encyklop{\"a}die
  der mathematischen Wissenschaften mit Einschlu{\ss} ihrer Anwendungen,
  Band IV, 2.~Teil, Heft 6}, Art.~32, pp.~3--90. B.~G. Teubner, Leipzig,
  1911.


\bibitem[Gal69]{Gallavotti1969}
G.~Gallavotti.
\newblock Divergences and the approach to equilibrium in the {L}orentz and the
  wind-tree models.
\newblock {\em Physical Review}, 185(1):308--322, 1969.
\newblock \url{https://doi.org/10.1103/PhysRev.185.308}.

\bibitem[Lor05]{Lorentz1905}
H.~A. Lorentz.
\newblock Le mouvement des {\'e}lectrons dans les m{\'e}taux.
\newblock {\em Archives n{\'e}erlandaises des sciences exactes et naturelles},
  series~2, 10:336--371, 1905.

\bibitem[LT20]{LT2020}
C.~Lutsko and B.~T{\'o}th.
\newblock Invariance principle for the random {L}orentz gas---beyond the
  {B}oltzmann--{G}rad limit.
\newblock {\em Communications in Mathematical Physics}, 379(2):589--632, 2020.
\newblock \url{https://doi.org/10.1007/s00220-020-03852-8}.

\bibitem[LT21]{LT2021}
C.~Lutsko and B.~T{\'o}th.
\newblock Invariance principle for the random wind-tree process.
\newblock {\em Annales Henri Poincaré}, 22(10):3357--3389, 2021.
\newblock \url{https://doi.org/10.1007/s00023-021-01106-4}.

  
\bibitem[MS10]{MarklofStrom2010}
J.~Marklof and A.~Str{\"o}mbergsson.
\newblock The distribution of free path lengths in the periodic {L}orentz gas
  and related lattice point problems.
\newblock {\em Annals of Mathematics}, 172(3):1949--2033, 2010.
\newblock \url{https://doi.org/10.4007/annals.2010.172.1949}.

\bibitem[MS11]{MarklofStrom2011}
J.~Marklof and A.~Str{\"o}mbergsson.
\newblock The {B}oltzmann--{G}rad limit of the periodic {L}orentz gas.
\newblock {\em Annals of Mathematics}, 174(1):225--298, 2011.
\newblock \url{https://doi.org/10.4007/annals.2011.174.1.7}.

\bibitem[MS14]{MarklofStrom2014}
J.~Marklof and A.~Str{\"o}mbergsson.
\newblock Power-law distributions for the free path length in {L}orentz gases.
\newblock {\em Journal of Statistical Physics}, 155(6):1072--1086, 2014.
\newblock \url{https://doi.org/10.1007/s10955-014-0935-9}.

\bibitem[MS24]{MarklofStrombergsson2024}
J.~Marklof and A.~Str{\"o}mbergsson.
\newblock Kinetic theory for the low-density {L}orentz gas.
\newblock {\em Memoirs of the American Mathematical Society},
  294(1464):v+136, 2024.
\newblock \url{https://doi.org/10.1090/memo/1464}.

\bibitem[Sha96]{Shah1996}
N.~A. Shah.
\newblock Limit distributions of expanding translates of certain orbits on
  homogeneous spaces.
\newblock {\em Proceedings of the Indian Academy of Sciences---Mathematical
  Sciences}, 106(2):105--125, 1996.
\newblock \url{https://doi.org/10.1007/BF02837164}.

\bibitem[Vor05]{Vorobets2005}
Y.~Vorobets.
\newblock On stability of periodic billiard orbits in polyhedra.
\newblock Unpublished manuscript, available at
  \url{https://people.tamu.edu/~yvorobets/Research/Stable.pdf}, 2005.

\end{thebibliography}
\end{document}